\documentclass{article} 
\usepackage{amsthm,amsmath,amsfonts,amssymb, authblk, hyperref,cases,color}
\usepackage{vmargin}
\setmarginsrb{2cm}{1cm}{2cm}{3cm}{1cm}{1cm}{2cm}{2cm}
\usepackage{tikz}
\usepackage[capitalise]{cleveref}
\usepackage[T1]{fontenc}
\usepackage[utf8]{inputenc}

\everymath{\displaystyle}

\numberwithin{equation}{section}
\newtheorem{Theorem}{Theorem}[section]
\newtheorem{Lemma}[Theorem]{Lemma}
\newtheorem{Proposition}[Theorem]{Proposition}

\newtheorem{Remark}[Theorem]{Remark}
\newtheorem{Notation}[Theorem]{Notation}

\renewcommand{\leq}{\leqslant}

\renewcommand{\geq}{\geqslant}

\DeclareMathOperator{\supp}{supp}
\DeclareMathOperator{\dist}{dist}
\DeclareMathOperator{\Span}{span}

\renewcommand{\div}{\operatorname{div}}
\renewcommand{\Re}{\operatorname{Re}}
\renewcommand{\Im}{\operatorname{Im}}

\begin{document}                  
\title{Controllability of a simplified fluid-structure interaction system}
\author[1]{R\'emi Buffe}
\author[1]{Tak\'eo Takahashi}
\affil[1]{Universit\'e de Lorraine, CNRS, Inria, IECL, F-54000 Nancy, France}
\date{\today}

\maketitle
	
\tableofcontents

\begin{abstract}
We are interested by the controllability of a fluid-structure interaction system where the fluid is viscous and incompressible and where the structure is elastic and located on a part of the boundary of the fluid's domain. In this article, we simplify this system by considering a linearization and by replacing the wave/plate equation for the structure by a heat equation. We show that the corresponding system coupling the Stokes equations with a heat equation at its boundary is null-controllable.
The proof is based on Carleman estimates and interpolation inequalities. One of the Carleman estimates corresponds to the case of Ventcel boundary conditions.
This work can be seen as a first step to handle the real system where the structure is modeled by the wave or the plate equation.
\end{abstract}

\vspace{1cm}



\section{Introduction}
Fluid-structure interaction systems are important systems for many applications such as aerodynamics, medicine (for instance the study of the motion of the blood in veins or in arteries), biology (animal locomotion in a fluid), civil engineering (design of bridges), naval architecture (design of boats and submarines), etc. Moreover, their mathematical studies can be very challenging due to several difficulties: in particular, the complexity of fluid equations such as the Navier-Stokes system, the strong coupling between the fluid system and the structure system and the free-boundary corresponding to the structure displacement. 

In this article we consider a simplified fluid-structure interaction system. 
The corresponding system without simplification has been proposed in  \cite{QuaTuvVen2000a} as a 
model for the blood flow in a vessel. It writes as follows: we denote by $\mathcal{I}$ the torus (in order to consider periodic boundary conditions):
$$
\mathcal{I}:=\mathbb{R}/(2\pi\mathbb{Z}),
$$
and for any deformation $\ell : \mathcal{I} \to (-1,\infty)$, we consider the corresponding fluid domain
\begin{equation}\label{gev0.0}
\Omega_{\ell}=  \left \{(x_1,x_2)\in \mathcal{I}\times \mathbb{R} \ ; \ x_2\in (0,1+\ell(x_1))\right \}.
\end{equation}
Then the equations of motion are
\begin{equation}\label{tak2.3}
\left\{\begin{array}{rl}
\partial_t w +(w\cdot \nabla) w-\div \mathbb{T}(w,\pi) = 0 & t>0, \  x\in \Omega_{\ell(t)},\\
\div w =   0 & t>0, \ x\in \Omega_{\ell(t)},\\
w(t,x_1,1+\ell(t,x_1)) = (\partial_t \ell)(t,x_1) e_2 & t>0, \ x_1\in \mathcal{I},\\
w = 0  & t>0, \ x\in \Gamma_{0},\\
\partial_{tt} \ell + \alpha_1 \partial_{x_1}^4 \ell-\alpha_2 \partial_{x_1}^2 \ell -\delta \partial_{t}\partial_{x_1}^2  \ell
  =-\widetilde{\mathbb{H}}_{\ell}(w,\pi) & t>0, \  x_1\in \mathcal{I},
\end{array}\right.
\end{equation}
where
$$
\Gamma_0=\mathcal{I}\times \{0\}, \quad 
\Gamma_1=\mathcal{I}\times \{1\}.
$$
In the above system, we have used the following notations: $(e_1, e_2)$ is the canonical basis of $\mathbb{R}^2$ and
\begin{equation}\label{ws4.6}
\mathbb{T}(w,\pi) = 2 D(w)- \pi I_2,\quad D(w)= \frac{1}{2}\left(\nabla w + (\nabla w)^*\right),
\end{equation}
\begin{equation}\label{ws4.7}
\widetilde{\mathbb{H}}_{\ell}(w,\pi)(t,x_1)= \left[ (1+|\partial_{x_1} \ell|^2)^{1/2} \left[\mathbb{T}(w,\pi)n\right](t,x_1,1+\ell(t,x_1))\cdot e_2 \right].
\end{equation}
The two first lines of \eqref{tak2.3} correspond to the Navier-Stokes system for the fluid velocity $w$ and the pressure $\pi$. The last line of \eqref{tak2.3} is a beam equation satisfied by the deformation $\ell$. We have used the standard no-slip boundary conditions (third and forth equations). To simplify, we assume that the viscosity of the fluid is constant and equal to 1. The vector fields $n$ corresponds to the unit exterior normal to $\Omega_{\ell(t)}$.

This system has been studied by many authors: \cite{ChaDesEst2005a} (existence of weak solutions), \cite{Bei2004a}, \cite{Leq2011a}, 
\cite{MR3466847} and \cite{Debayan} (existence of strong solutions), \cite{RaymondSICON2010} (stabilization of strong solutions), \cite{MR3619065} (stabilization of weak solutions around a stationary state). There are also some works in the case $\delta=0$, that is  without damping on the beam equation: the existence of weak solutions is proved in \cite{Gra2008a} and in \cite{MuhaCanic} (see also \cite{CanicMuhaBukac}). In \cite{grandmont:hal-01567661}, 
the existence of local strong solutions is obtained for a structure described by either a wave equation ($\alpha_1=\delta=0$ and $\alpha_2>0$) 
or a beam equation with inertia of rotation ($\alpha_1>0$, $\alpha_2=\delta=0$ and with an additional term $-\partial_{ttss}\ell$). 
In \cite{plat} and \cite{nonplat}, the authors show the existence and uniqueness of strong solutions in the case  $\alpha_1>0$, $\alpha_2\geq 0$ and $\delta=0$.
Using similar techniques they also analyze the case of the wave equation ($\alpha_1=\delta=0$ and $\alpha_2>0$) in \cite{wave} showing in particular that the semigroup of the linearized system is analytic.  
Let us mention also some results for more complex models: \cite{LR14, Leng14} (linear elastic Koiter shell),  \cite{MC-arma} (dynamic pressure boundary conditions), \cite{Boris-C-13, Boris-C-15} (3D cylindrical domain with nonlinear elastic cylindrical Koiter shell), 
\cite{MR4042350} and \cite{MR4192391}  (nonlinear elastic and thermoelastic plate equations),  \cite{MR4253566}, \cite{MR4183914} (compressible fluids), etc.

The advantage of the damping in the beam equation is that the term $-\delta \partial_{t}\partial_{x_1}^2  \ell$ is a structural damping so that the corresponding beam equation becomes a parabolic equation 
(see, for instance, \cite{CheTri1989a}). In this work, we consider a simplified model associated with \eqref{tak2.3}. We neglect the deformation of the fluid domain due to the elastic déformation and we also linearized the Navier-Stokes system by only considering the Stokes system. Moreover, we replace the damped beam equation by a heat equation. By setting 
$$
\Omega=\mathcal{I}\times (0,1)
$$
(see Figure \ref{F1}), we are thus considering the following system
\begin{equation}\label{ns0.1}
\left\{\begin{array}{rl}
\partial_{t} w - \Delta w + \nabla \pi = 1_{\omega} f & \text{in} \ (0,T)\times \Omega,\\
\div w=0& \text{in} \ (0,T)\times \Omega,\\
w=0 & \text{on} \ (0,T)\times \Gamma_0,\\
w=\zeta e_2 & \text{on} \ (0,T)\times \Gamma_1,\\
\partial_t \zeta -\partial_{x_1x_1} \zeta = -\mathbb{T}(w,\pi)n\cdot e_2 & \text{in}\  (0,T)\times \mathcal{I},\\
w(0,\cdot)=w^{0} \quad \text{in} \ \Omega, & \zeta(t,0)=\zeta^0 \quad \text{in} \ \mathcal{I}.
\end{array}\right.
\end{equation}
In the above system, $\zeta$ corresponds to the displacement velocity $\partial_t\ell$ in \eqref{tak2.3} and we do not consider anymore the displacement position $\ell$. We have added a control $f$ localized in the fluid domain, in an arbitrary small nonempty open set $\omega$ of $\Omega$. Our goal is to show the null-controllability of this simplified system and to do this, as it is standard (see, for instance, \cite[Theorem 11.2.1, p.357]{TucsnakWeiss}), we prove an observability inequality on the adjoint system:
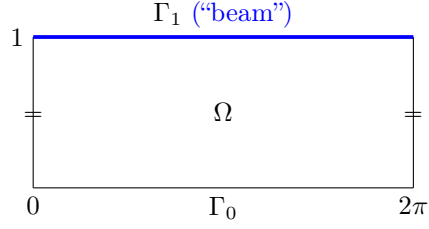
\begin{figure}\label{F1}
\begin{center}
\begin{tikzpicture}
\draw (0,0) -- (5,0);
\draw (0,2) -- (0,0);
\draw (5,0) -- (5,2);
\draw[line width=0.5mm,  blue] (0,2) -- (5,2);
\draw (2.5,1) node {$\Omega$};
\draw (2.5,0) node[below] {$\Gamma_{0}$};
\draw (0,1) node {$-$};
\draw (0,0.95) node {$-$};
\draw (5,1) node {$-$};
\draw (5,0.95) node {$-$};
\draw (2.5,2) node[above] {$\Gamma_{1}$ {\color{blue}(``beam'')}};
\draw (0,0) node[below] {$0$};
\draw (5,0) node[below] {$2\pi$};
\draw (0,2) node[left] {$1$};
\end{tikzpicture}
\end{center}
\caption{Our geometry}
\end{figure}
\begin{equation}\label{ns0.2}
\left\{\begin{array}{rl}
\partial_{t} u - \Delta u + \nabla p = 0& \text{in} \ (0,T)\times \Omega,\\
\div u=0& \text{in} \ (0,T)\times \Omega,\\
u=0 & \text{on} \ (0,T)\times\Gamma_0,\\
u=\eta e_2 & \text{on} \ (0,T)\times \Gamma_1,\\
\partial_t \eta -\partial_{x_1x_1} \eta = -\mathbb{T}(u,p)n_{|\Gamma_1}\cdot e_2 & \text{in}\  (0,T)\times \mathcal{I},\\
u(0,\cdot)=u^0 \quad \text{in} \ \Omega, & \eta(0,\cdot)=\eta^0 \quad \text{in} \ \mathcal{I}.
\end{array}\right.
\end{equation}
We set
$$
L^2_0(\mathcal{I}):=\left\{ f\in L^2(\mathcal{I}) \ ; \ \int_0^{2\pi} f(x_1)\ dx_1=0\right\}
$$
and we define the space
\begin{equation}\label{gev7.4}
\mathcal{H} :=  \left\{\left[u,\eta \right]\in {L}^2(\Omega)\times L^2_0(\mathcal{I})
\ ;\ u_2=0 \ \text{on} \ \Gamma_0, \quad 
u_2=\eta  \ \text{on} \ \Gamma_1,\ \div u =0 \; \text{in} \ \Omega \right\}.
\end{equation}
Then, our main result states as follows:
\begin{Theorem}\label{Tmain}
Let $\gamma>1$ and let $\omega$ be a nonempty open set of $\Omega$. 
Then, there exists $C_0>0$ such that for any $T\in (0,1)$ and for any $\left[u^0,\eta^0\right]\in \mathcal{H}$,
the solution $\left[u,\eta \right]$ of \eqref{ns0.2} satisfies 
$$
\left\|\left[u(T,\cdot),\eta(T,\cdot) \right]\right\|_{\mathcal{H}}^2 \leq 
C_0 \exp\left(\frac{C_0}{T^\gamma}\right)
\iint_{(0,T)\times \omega} |u|^2 \ dt dx.
$$
In particular, for any $\left[w^0,\zeta^0\right]\in \mathcal{H}$ and for any $T>0$, there exists a control $f\in L^2((0,T)\times \omega)$ such that the solution 
$\left[w,\zeta \right]$ of \eqref{ns0.1} satisfies 
$$
w(T,\cdot)=0, \quad \zeta(T,\cdot)=0.
$$
\end{Theorem}
\begin{Remark}
As explained above, \cref{Tmain} can be seen as a first control result on a simplified model. We expect to extend some of the tools developed here to handle the control properties of the system \eqref{tak2.3} in future works. The controllability properties of fluid-structure interaction systems have been tackled mainly in the case where the structure is a rigid body (see, \cite{MR4101928}, \cite{MR2139944}, \cite{MR3023058}, \cite{MR2317341}, \cite{MR2375750}, \cite{MR3085093}, \cite{MR4043319}, \cite{djebour:hal-02454567}, \cite{MR3365831}, etc.). In \cite{MR4222171}, the author shows an observability inequality for the adjoint of a linearized and simplified fluid-structure interaction system in the case of a compressible viscous fluid and of a damped beam. Note that the corresponding control problem involves two controls, one for the fluid and one for the structure.

For the stabilization of fluid-structure interaction systems, one can quote some results: 
\cite{MR2745779}, \cite{MR3619065} (for the case of a damped beam), \cite{MR3397328}, \cite{MR4238201}, \cite{MR3181675}, \cite{MR3261920} (for the case of a rigid body).
\end{Remark}

\begin{Remark}
The method proposed here to control a system involving the Stokes equations is quite different from  the method used in a large part of the literature for the controllability of fluid systems. In general, the method is based on ``global Carleman inequalities'' (see, for instance \cite{fernandez2004local, imanuvilov1998}). Here, we follow another strategy as in \cite{LebeauRobbiano,ChavesLebeau}. Such a method is based on local Carleman inequalities for an ''augmented'' elliptic operator, from which one deduces a spectral inequality, in the spirit of \cite{LebeauRobbiano,JerisonLebeau}. However, as it is pointed out in \cite{ChavesLebeau}, unique continuation property does not hold for the augmented operator in the direction of the additional variable, due to the pressure. We then use an adaptation of the original strategy of \cite{LebeauRobbiano,Leautaud} in our context. This type of spectral inequality has already been used in the context of fluids in \cite{BuffeGagnon}. We also recall that one can use \cref{Tmain} to handle nonlinear controllability issues by applying the general method proposed in \cite{MR3023058}.
\end{Remark}

\begin{Remark}
Using the particular geometry considered here, we can simplify the adjoint system.
First on $\Gamma_1$, $n=e_2$ and using \eqref{ws4.6}, we deduce
\begin{equation}\label{ns0.3}
-\mathbb{T}(u,p)n\cdot e_2= -2\partial_{x_2} u_2 +p=2\partial_{x_1} u_1 +p=p,
\end{equation}
since $u_1(x_1,1)=0$ for $x_1\in \mathcal{I}$.

Moreover, using the incompressibility of the fluid and the boundary conditions, we deduce that
$$
0=\int_\Omega \div u \ dx =\int_{\mathcal{I}} \eta \ dx_1.
$$
Using this condition on the heat equation on the boundary and \eqref{ns0.3} yields
\begin{equation}\label{meanpressure}
\int_{\mathcal{I}}  p(x_1,1)\ dx_1=0.
\end{equation}
In particular, in contrast with the standard Stokes system, the pressure is not determined up to a constant. 
\end{Remark}

The outline of the article is as follows: in \cref{Sec_gen}, we show how to obtain the observability inequality stated in \cref{Tmain} from a spectral inequality. Such a result is quite standard, but here we show that for a self-adjoint operator, we do not need the usual assumption that is made on the eigenvalues of the main operator. 
Then using this general result, we are reduced to show a spectral inequality that we state in \cref{Sec_fun} along with the functional framework. The spectral inequality is itself the consequence of an interpolation inequality that we obtain in \cref{Sec_spectral}. One of the main difficulties to obtain such an inequality comes from the fact that we need to estimate the pressure. \cref{Sec_pressure} is devoted to such an estimate which is one of the main parts of this article. The proof of the spectral inequality and thus of \cref{Tmain} is obtained at the end of \cref{Sec_spectral}. In the appendix, we show an interpolation estimate for the Ventcel boundary condition that is mainly a consequence of a Carleman estimate obtained in \cite{Buffe17}.

\begin{Notation}
In the whole paper, we use $C$ as a generic positive constant that does not depend on the other terms of the inequality.
The value of the constant $C$ may change from one appearance to another.
We also use the notation $X\lesssim Y$ if there exists a constant $C>0$ such that we have the inequality $X\leq CY$. 
\end{Notation}

\section{From a spectral inequality to the null-controllability}\label{Sec_gen}
This section is devoted to a ``classical'' result showing that a spectral inequality implies the final-state observability and thus the null-controllability. 
The proof follows closely the proof in \cite{Leautaud, LebeauRobbiano} and we only prove it here for sake of completeness and also to show that we do not need 
any assumption on the asymptotic behavior of the spectrum of the operator (which is used in the above references).

More precisely, we assume here that $A : \mathcal{D}(A)\to \mathcal{H}$ is a positive self-adjoint operator with compact resolvents in an Hilbert space $\mathcal{H}$. We denote by $(\lambda_j)$ the nondecreasing sequence of eigenvalues and by $(w_j)$ an orthonormal basis of $\mathcal{H}$ composed by eigenvectors of $A$:
$A w_j =\lambda_j w_j$ for $j\geq 1$.
We also consider a control operator $B\in \mathcal{L}(\mathcal{U},\mathcal{H}).$


\begin{Theorem}\label{T2}
Assume the above hypotheses. Assume moreover
the existence of $S_0>0$, $C>0$ and $\kappa\in C^{\infty}_0(0,S_0)$ such that
that for any $\Lambda>0$, and for any $(a_{j})_{j}\in \mathbb C^{\mathbb N}$,
\begin{equation}\label{abs1.0}
\sum_{\lambda_{j}\leq \Lambda} \left| a_j \right|^2
\leq 
Ce^{C\sqrt{\Lambda}} \int_{0}^{S_{0}} \kappa^2(s)
\left\| \sum_{\lambda_{j}\leq \Lambda} a_j \cosh(s\sqrt\lambda_{j})
B^* w^{(j) }
 \right\|_{\mathcal{U}}^{2} \ ds.
\end{equation}
Then for all $\gamma>1$, there exists $C_0>0$ such that for any $T\in (0,1)$ and for any $z^0\in \mathcal{H}$,
\begin{equation}\label{abs4.7}
\left\|e^{-TA}z^0\right\|_{\mathcal{H}}^2 \leq 
C_0 \exp\left(\frac{C_0}{T^\gamma}\right)
\int_{0}^{T} \left\| B^* e^{-tA}z^0 \right\|^2_{\mathcal{U}} \ dt.
\end{equation}
\end{Theorem}
We recall that relation \eqref{abs4.7} implies the null-controllability of the system
\begin{equation}\label{abs4.9}
\left\{
\begin{array}{l}
\frac{d \theta}{dt} +A \theta=B g \quad \text{in} \ (0,T),
\\[2mm]
\theta(0)=\theta^0\in \mathcal{H}.
\end{array}
\right.
\end{equation}

\subsection{Controllability of the first modes}
We define
$$
\mathcal{H}_{\Lambda} = \Span \left\{ w_j, \ \lambda_j\leq \Lambda\right\},
\quad \Pi_{\Lambda} : \mathcal{H} \to \mathcal{H}_{\Lambda} \ \text{the orthogonal projection}.
$$
We are interested here by the control problem
\begin{equation}\label{abs3.1}
\left\{
\begin{array}{l}
\frac{d \theta}{dt} +A \theta=\Pi_{\Lambda} B g \quad \text{in} \ (0,\tau),
\\[2mm]
\theta(0)=\theta^0\in \mathcal{H}_{\Lambda},
\end{array}
\right.
\end{equation}
for some $\tau>0$.
We consider the linear operator
$$
G_{\Lambda} :=\int_0^{S_0} \kappa^2(s) \cosh(s\sqrt{A})\Pi_{\Lambda} B B^* \cosh(s\sqrt{A})\Pi_{\Lambda} \ ds.
$$
From \eqref{abs1.0}, $G_{\Lambda} \in \mathcal{L}(\mathcal{H}_{\Lambda})$ is symmetric, positive and invertible with
$$
\left\| G_{\Lambda}^{-1} \right\|_{\mathcal{L}(\mathcal{H}_{\Lambda})}\leq Ce^{C\sqrt{\Lambda}}.
$$
We set
$$
\sigma:=2-\frac{1}{\gamma}\in (1,2).
$$
From \cite[Lemma A.1]{Leautaud}, 
there exists $e\in C^\infty(\mathbb{R})$ such that for some constants $c_j$
\begin{equation}\label{abs2.0}
\supp e=[0,1], 
\end{equation}
\begin{equation}\label{abs2.1}
\left| \widehat e (z) \right| \leq c_1 e^{-c_2 |z|^{1/\sigma}} \quad \text{if} \ \Im(z)\leq 0,
\end{equation}
\begin{equation}\label{abs2.2}
\left| \widehat e (z) \right| \geq c_3 e^{-c_4 |z|^{1/\sigma}} \quad \text{if} \ z\in i\mathbb{R}^-.
\end{equation}
From \eqref{abs2.2}, we have that
$\widehat e(-i\tau A)\in  \mathcal{L}(\mathcal{H}_{\Lambda})$ is invertible and
\begin{equation}\label{abs2.3}
\left\| \widehat e(-i\tau A)^{-1} \right\|_{\mathcal{L}(\mathcal{H}_{\Lambda})}\leq \frac{1}{c_3}e^{c_4(\Lambda \tau)^{1/\sigma}}.
\end{equation}
Then we define 
$$
h_{\Lambda}(s):=-\frac{1}{2}\kappa^2(s) B^* \cosh(s\sqrt{A}) G_{\Lambda}^{-1} \widehat e(-i\tau A)^{-1} e^{-\tau A} \theta^0 \quad (s\in \mathbb{R}).
$$
We have that $h_{\Lambda}\in C^\infty_0(\mathbb{R},\mathcal{U})$ with $\supp h_{\Lambda}\subset (0,S_0)$ and
\begin{equation}\label{abs2.4}
\left\| h_{\Lambda} \right\|_{L^\infty(\mathbb{R},\mathcal{U})}\leq Ce^{C\sqrt{\Lambda}+c_4(\Lambda \tau)^{1/\sigma}}
	\left\| \theta^0 \right\|_{\mathcal{H}}.
\end{equation}
Thus, $\widehat h_{\Lambda}\in \mathrm{Hol}(\mathbb{C};\mathcal{U})$ and 
\begin{equation}\label{abs2.5}
\left\| \widehat h_{\Lambda}(z) \right\|_{\mathcal{U}}\leq Ce^{C\sqrt{\Lambda}+c_4(\Lambda \tau)^{1/\sigma}}e^{S_0 \left| \Im(z)\right|}
	\left\| \theta^0 \right\|_{\mathcal{H}}.
\end{equation}
As in \cite{Russell}, we introduce $Q_{\Lambda}\in \mathrm{Hol}(\mathbb{C};\mathcal{U})$ such that
$$
Q_{\Lambda}(-iz^2)=\widehat h_{\Lambda}(iz)+\widehat h_{\Lambda}(-iz) \quad (z\in \mathbb{C}).
$$
We deduce from the above relation and \eqref{abs2.5} that
\begin{equation}\label{abs2.6}
\left\| Q_{\Lambda}(z) \right\|_{\mathcal{U}}\leq Ce^{C\sqrt{\Lambda}+c_4(\Lambda \tau)^{1/\sigma}}e^{S_0 \sqrt{\left|z\right|}}
	\left\| \theta^0 \right\|_{\mathcal{H}}.
\end{equation}
We define $\mathfrak{g}_{\Lambda}\in \mathrm{Hol}(\mathbb{C};\mathcal{U})$ by
$$
\mathfrak{g}_{\Lambda}(z) := \widehat e(\tau z) Q_{\Lambda}(z).
$$
From \eqref{abs2.0}, \eqref{abs2.1} and \eqref{abs2.6}, we have
\begin{equation}\label{abs2.7}
\left\| \mathfrak{g}_{\Lambda}(z) \right\|_{\mathcal{U}}\leq Ce^{C\sqrt{\Lambda}+c_4(\Lambda \tau)^{1/\sigma}}e^{S_0 \sqrt{\left|z\right|}}e^{\tau|\Im z|}
	\left\| \theta^0 \right\|_{\mathcal{H}} \quad (z\in \mathbb{C})
\end{equation}
and
\begin{equation}\label{abs2.8}
\left\| \mathfrak{g}_{\Lambda}(z) \right\|_{\mathcal{U}}\leq Ce^{C\sqrt{\Lambda}+c_4(\Lambda \tau)^{1/\sigma}}e^{S_0 \sqrt{\left|z\right|}}e^{-c_2 \tau^{1/\sigma}|z|^{1/\sigma}}
	\left\| \theta^0 \right\|_{\mathcal{H}} \quad \text{if} \ \Im z\leq 0.
\end{equation}
Since $\sigma<2$, we can use a Paley-Wiener type theorem (see \cite[Proposition A.3]{Leautaud}) and deduce the existence of $g_\Lambda\in C^\infty_0((0,\tau);\mathcal{U})$ such that
$$
\widehat g_\Lambda(z) = \mathfrak{g}_{\Lambda}(z).
$$
In particular, from \eqref{abs2.8} and the Laplace method, for all $t\in (0,\tau)$,
\begin{equation}\label{abs3.2}
\left| g_\Lambda(t) \right| \leq \left\|\mathfrak{g}_{\Lambda} \right\|_{L^1(\mathbb{R};\mathcal{U})}
\leq Ce^{C\sqrt{\Lambda}+c_4(\Lambda \tau)^{1/\sigma}+\frac{C}{\tau^{1/(2-\sigma)}}}
	\left\| \theta^0 \right\|_{\mathcal{H}}.
\end{equation}
Now, for any $j$ such that $\lambda_j\leq \Lambda$,
\begin{multline*}
\left(\int_0^\tau e^{-(\tau-t)A}Bg_\Lambda(\tau-t) \ dt,w_j\right)_{\mathcal{H}}
=\left(B\widehat g_\Lambda(-i \lambda_j),w_j\right)_{\mathcal{H}}
\\
=\left( \widehat e(-i \tau A) B \left(\widehat h_{\Lambda}(i\sqrt{\lambda_j})+\widehat h_{\Lambda}(-i\sqrt{\lambda_j})\right), w_j\right)_{\mathcal{H}}
\\
=\left( \widehat e(-i \tau A) B \int_0^{S_0} h_\Lambda (s) 2 \cosh(s\sqrt{\lambda_j})\ ds, w_j\right)_{\mathcal{H}}
=-\left( e^{-\tau A} \theta^0, w_j\right)_{\mathcal{H}}
\end{multline*}
so that the solution $\theta$ of \eqref{abs3.1} with the control $g(t)=g_\Lambda(\tau-t)$ satisfies $\theta(\tau)=0$.

By a duality argument and \eqref{abs3.2}, this implies that
\begin{equation}\label{abs4.0}
\left\| \Pi_\Lambda e^{-\tau A} z^0 \right\|^2 \leq C_1 \exp\left(C_1\left(\frac{1}{\tau^\gamma}+\sqrt{\Lambda}+(\Lambda \tau)^{\gamma/(2\gamma-1)}\right)\right) 
\int_{0}^{\tau} \left\| B^* \Pi_\Lambda e^{-t A} z^0 \right\|^2_{\mathcal{U}} \ dt.
\end{equation}

\subsection{Proof of \cref{T2}}
We are now in a position to prove \cref{T2}, adapting the method of \cite{Miller}.
\begin{proof}[Proof of \cref{T2}]
We set
$$
z(t)=e^{-tA} z^0.
$$

Assume
$$
T^{(1)}>0, \quad \tau>0, \quad \text{and} \quad T^{(2)}=T^{(1)}+\tau.
$$
From \eqref{abs4.0}
\begin{equation}\label{abs0.0}
\left\| \Pi_\Lambda z\left(T^{(2)}\right)\right\|^2 \leq C_1 \exp\left(C_1\left(\frac{1}{\left(\varepsilon \tau \right)^\gamma}+\sqrt{\Lambda}+(\Lambda \varepsilon \tau)^{\gamma/(2\gamma-1)}\right)\right) 
\int_{T^{(2)}-\varepsilon \tau }^{T^{(2)}} \left\| B^* \Pi_\Lambda z(t) \right\|^2_{\mathcal{U}} \ dt.
\end{equation}

We set
\begin{equation}\label{LambdaT}
\Lambda = \frac{1}{\left(\varepsilon \tau\right)^{1+\gamma}}
\end{equation}
so that for
$$
\tau, \varepsilon \in (0,1),
$$
\eqref{abs0.0} becomes
\begin{equation}\label{abs0.1}
2\rho(\tau) \left\| \Pi_\Lambda z\left(T^{(2)}\right)\right\|^2 \leq 
\int_{T^{(2)}-\varepsilon \tau }^{T^{(2)}} \left\| B^* \Pi_\Lambda z(t) \right\|^2_{\mathcal{U}} \ dt
\end{equation}
with
\begin{equation}\label{abs0.2}
\rho(\tau):=\frac{1}{2C_1} \exp\left(-\frac{3 C_1}{\left(\varepsilon \tau\right)^\gamma}\right).
\end{equation}
Then from \eqref{abs0.1}, we deduce
\begin{equation}\label{abs0.3}
\rho(\tau) \left\| z\left(T^{(2)}\right)\right\|^2 \leq 
\int_{T^{(2)}-\varepsilon \tau }^{T^{(2)}} \left\| B^* z(t) \right\|^2_{\mathcal{U}} \ dt
+
C \left\|z\left(T^{(1)}\right)\right\|_{\mathcal{H}}^2  \varepsilon \tau e^{-2\Lambda \tau(1-\varepsilon)}
+\rho(\tau) \left\|z\left(T^{(1)}\right)\right\|_{\mathcal{H}}^2 e^{-2\Lambda \tau}
\end{equation}
For $\varepsilon>0$ small enough, the above relation yields
\begin{equation}\label{abs0.4}
\rho(\tau) \left\|z\left(T^{(2)}\right)\right\|_{\mathcal{H}}^2 \leq 
\int_{T^{(1)}}^{T^{(2)}} \left\| B^* z(t) \right\|^2_{\mathcal{U}} \ dt
+\rho(\tau/2) \left\|z\left(T^{(1)}\right)\right\|_{\mathcal{H}}^2
\end{equation}

Assume
$$
T\in (0,1).
$$
Then for all $k\geq 0$, \eqref{abs0.4} implies
\begin{equation}\label{abs0.5}
\rho\left(\frac{T}{2^{k+1}}\right) \left\|z\left(\frac{T}{2^{k}}\right)\right\|_{\mathcal{H}}^2 \leq 
\int_{\frac{T}{2^{k+1}}}^{\frac{T}{2^{k}}} \left\| B^* z(t) \right\|^2_{\mathcal{U}} \ dt
+\rho\left(\frac{T}{2^{k+2}}\right) \left\|z\left(\frac{T}{2^{k+1}}\right)\right\|_{\mathcal{H}}^2
\end{equation}
and thus
\begin{equation}\label{abs0.6}
\rho\left(\frac{T}{2}\right) \left\|z\left(T\right)\right\|_{\mathcal{H}}^2 \leq 
\int_{0}^{T} \left\| B^* z(t) \right\|^2_{\mathcal{U}} \ dt.
\end{equation}
Thus for some constant $C_2>0$, 
\begin{equation}\label{abs0.7}
\left\|z\left(T\right)\right\|_{\mathcal{H}}^2 \leq 
C_2 \exp\left(\frac{C_2}{T^\gamma}\right)
\int_{0}^{T} \left\| B^* z(t) \right\|^2_{\mathcal{U}} \ dt.
\end{equation}

\end{proof}

\section{Functional framework and spectral inequality}\label{Sec_fun}
In order to prove \cref{Tmain}, we are going to apply \cref{T2}.
In this section, we first give the functional framework associated with \eqref{ns0.2}. 
Then we write the spectral inequality that will be proven in the remaining part of the article.
\subsection{Functional framework}
We recall that $\mathcal{H}$ is defined by \eqref{gev7.4}. We also define
$$
\mathcal{V} :=  \left\{\left[u,\eta\right]\in  \left({H}^1(\Omega)\times H^1(\mathcal{I})\right) \cap \mathcal{H}
 \ ; \ u_1=0 \ \text{on} \ \partial\Omega\right\}.
$$
We denote by $P_0$ the orthogonal projection from ${L}^2(\Omega)\times L^2_0(\mathcal{I})$ onto $\mathcal{H}$. 
We now define the linear operator 
$A_0 : \mathcal{D}(A_0)\subset \mathcal{H}\to \mathcal{H}$ by
\begin{equation}\label{fs3.4}
\mathcal{D}(A_0):=   \mathcal{V} \cap  \Big [{H}^2(\Omega)\times H^2(\mathcal{I})\Big],
\end{equation}
and for $\begin{bmatrix} u,\eta \end{bmatrix}\in \mathcal{D}(A_0)$, we set
\begin{equation}\label{fs3.2}
{A}_0\begin{bmatrix} u \\ \eta \end{bmatrix} := 
P_0\begin{bmatrix}
 \Delta u \medskip\\ \medskip
 \displaystyle \partial_{x_1}^2 \eta
\end{bmatrix}.
\end{equation}
Then one can check that \eqref{ns0.2} writes 
$$
\left\{\begin{array}{rl}
\frac{dz}{dt} = A_0 z \quad \text{in} \ (0,T), \\
z(0)=z^0.
\end{array}\right.
$$
with $z=[u,\eta]$, $z^0=[u^0,\eta^0]$.
In the next proposition, we show in particular that $A_0$ is the infinitesimal generator of a semigroup so that $z(t)=e^{tA_0}z^0$ for $t\geq 0$.
\begin{Proposition}\label{AStrongContSG}
The operator $A_0$ defined by \eqref{fs3.4}--\eqref{fs3.2} has compact resolvents, 
and is self-adjoint negative on $\mathcal{H}$.
\end{Proposition}
\begin{proof}
By definition of $A_0$, we have for any $[u,\eta], [v,\zeta]\in \mathcal{D}(A_0)$, 
$$
\left\langle A_0 
\begin{bmatrix}
u\\ \eta
\end{bmatrix},
\begin{bmatrix}
v\\ \zeta
\end{bmatrix}
 \right\rangle_{\mathcal{H}}
=\int_{\Omega} \Delta u \cdot v \ dx +\int_0^{2\pi} \left(\partial_{x_1}^2 \eta\right) \zeta \ dx_1
=-\int_{\Omega} \nabla u : \nabla v \ dx -\int_0^{2\pi} \left(\partial_{x_1} \eta\right)\left(\partial_{x_1} \zeta\right) \ dx_1.
$$
Thus $A_0$ is symmetric and negative (by using the Poincar\'e inequalities).

In order to show that $A_0$ is self-adjoint it is sufficient to show that it is onto. Assume $[f, g]\in \mathcal{H}$ and let us solve the equation
\begin{equation}\label{onto}
-A_0 \begin{bmatrix}
u\\ \eta
\end{bmatrix}=[f, g].
\end{equation}
Multiplying the above equation by $[v,\zeta]\in \mathcal{V}$ leads to the weak formulation
\begin{equation}\label{weakform}
\int_{\Omega} \nabla u : \nabla v \ dx +\int_0^{2\pi} \left(\partial_{x_1} \eta\right)\left(\partial_{x_1} \zeta\right) \ dx_1
=\int_{\Omega} f\cdot v \ dx+\int_0^{2\pi} g \zeta \ dx_1 \quad ([v,\zeta]\in \mathcal{V}).
\end{equation}
Using the Poincaré inequalities, we see that we can apply the Riesz theorem and deduce the existence and uniqueness of
$[u,\eta]\in \mathcal{V}$ solution of \eqref{weakform}. Then if $v\in C^\infty_c(\Omega)$ with $\div v=0$ and $\zeta=0$ in \eqref{weakform}, we obtain that
\begin{equation}\label{weakform2}
\int_{\Omega} \nabla u : \nabla v \ dx
=\int_{\Omega} f\cdot v \ dx \quad (v\in C^\infty_c(\Omega), \ \div v=0).
\end{equation}
Using the De Rham theorem, we deduce the existence of $p$ such that
\begin{equation}\label{Stokes}
\begin{cases}
- \Delta u + \nabla p = f \quad \text{in} \ \Omega,\\
\div u=0  \quad \text{in} \ \Omega,\\
u=0 \quad \text{on} \ \Gamma_0\\
u=\eta e_2 \quad \text{on} \ \Gamma_1.
\end{cases}
\end{equation}
Using the elliptic regularity of the Stokes system, we deduce that $(u,p)\in H^{3/2}(\Omega)\times H^{1/2}(\Omega)$.
Multiplying the first above equation by $v$, with $[v,\zeta]\in \mathcal{V}$, we deduce that for any $\zeta\in H^1(\mathcal{I})\cap L^2_0(\mathcal{I})$,
\begin{equation}\label{Beamheat}
\int_0^{2\pi} \left(\partial_{x_1} \eta\right)\left(\partial_{x_1} \zeta\right) \ dx_1
=-\langle p_{|\Gamma_1},\zeta \rangle+\int_0^{2\pi} g \zeta \ dx_1.
\end{equation}
Since $p_{|\Gamma_1}=\left(\left[pI_3-\nabla u\right]n\right)_{|\Gamma_1}\in H^{-1/2}(\mathcal{I})$, we deduce that $\eta\in \mathcal{H}^{3/2}(\mathcal{I})$ and from \eqref{Stokes} that $(u,p)\in H^2(\Omega)\times H^1(\Omega)$. Thus $p_{|\Gamma_1}\in H^{1/2}(\mathcal{I})$, and from \eqref{Beamheat}, we deduce that
$\eta\in H^2(\mathcal{I})$. We conclude that $[u,\eta]\in \mathcal{D}(A_0)$ and satisfies \eqref{onto}.

The fact that $A_0$ has compact resolvents is coming from the compact embedding of $H^2$ into $L^2$ for bounded domains.
\end{proof}

In particular, the eigenvalues $\lambda_j>0$ of $-A_0$ satisfy $\lambda_j\to \infty$ and there exists
\begin{equation}\label{ns7.1}
\left(\begin{bmatrix}
u^{(j)}\\ \eta^{(j)}
\end{bmatrix}\right)_j \quad \text{orthonormal basis of} \quad \mathcal{H}
\end{equation}
composed by eigenvectors of $A_0$:
\begin{equation}\label{ns7.0}
-A_0 \begin{bmatrix}
u^{(j)}\\ \eta^{(j)}
\end{bmatrix}=\lambda_j
\begin{bmatrix}
u^{(j)}\\ \eta^{(j)}
\end{bmatrix}
\end{equation}
The above system can be written as
\begin{equation}\label{ns4.0}
\begin{cases}
 - \Delta u^{(j)} + \nabla p^{(j)} = \lambda_j u^{(j)}\\
\div u^{(j)}=0\\
u^{(j)}=0 \quad \text{on} \ \Gamma_0\\
u^{(j)}=\eta^{(j)} e_2 \quad \text{on} \ \Gamma_1\\
 -\partial_{x_1}^2 \eta^{(j)} - p^{(j)}=\lambda_j \eta^{(j)} \quad \text{in}\  \mathcal{I}
\end{cases}
\end{equation}
and more precisely as
\begin{equation}\label{ns4.0-bis}
\begin{cases}
-\lambda_j u_1^{(j)} - (\partial_{x_1}^2+\partial_{x_2}^2) u_1^{(j)} + \partial_{x_1} p^{(j)} = 0\\
-\lambda_j u_2^{(j)} - (\partial_{x_1}^2+\partial_{x_2}^2) u_2^{(j)} + \partial_{x_2} p^{(j)} = 0\\
\partial_{x_1} u_1^{(j)}+\partial_{x_2} u_2^{(j)}=0\\
u_1^{(j)}=0 \quad \text{on} \ \partial \Omega\\
u_2^{(j)}=0 \quad \text{on} \ \Gamma_0\\
-\lambda_j u_2^{(j)} -\partial_{x_1}^2 u_2^{(j)} = p^{(j)} \quad \text{on} \ \Gamma_1.
\end{cases}
\end{equation}

\subsection{Spectral inequality}
We are now in a position to state the spectral inequality for the operator $A_0$ defined in the previous section.
\begin{Theorem}\label{SpectralInequality}
Let $\omega_0$ be a nonempty open subset of $\Omega$ and $S_0>0$. 
There exist $C>0$ and $\kappa\in C^{\infty}_{0}(0,S_{0})$ such that for any $\Lambda>0$, and for any $(a_{j})_{j}\in \mathbb C^{\mathbb N}$,
\begin{equation}\label{ns4.1}
\sum_{\lambda_{j}\leq \Lambda} \left| a_j \right|^2
\leq 
Ce^{C\sqrt{\Lambda}} \int_{0}^{S_{0}} \kappa^2(s)
\left\| \sum_{\lambda_{j}\leq \Lambda} a_j \cosh(s\sqrt\lambda_{j})
u^{(j) }
 \right\|_{L^{2}(\omega)}^{2} \ ds.
\end{equation}
\end{Theorem}
In order to prove \cref{SpectralInequality}, we define for $s\in (0,S_0)$ and $x\in \Omega$,
\begin{equation}\label{DefofUandP}
U(s,x):=\sum_{\lambda_{j}\leq \Lambda} a_{j}\cosh(\sqrt \lambda_{j}s)u^{(j)}(x),\quad 
P(s,x):= \sum_{\lambda_{j}\leq \Lambda} a_{j}\cosh(\sqrt \lambda_{j}s) p^{(j)}(x)+c_P(s)
\end{equation}
and the domains
\begin{equation}\label{defZJ}
Z:=(0,S_{0})\times \Omega=(0,S_{0})\times \mathcal{I}\times (0,1), \quad
J_i:=(0,S_{0})\times \Gamma_i=(0,S_{0})\times \mathcal{I}\times \{i\}
\quad
(i=0,1).
\end{equation}
From \eqref{ns4.0-bis}, we deduce that 
\begin{equation}\label{ns2.0.1}
\left\{\begin{array}{l}
-\partial_{s}^2 U_1 - (\partial_{x_1}^2+\partial_{x_2}^2) U_1 + \partial_{x_1} P = 0 \quad \text{in} \ Z,\\
-\partial_{s}^2 U_2 - (\partial_{x_1}^2+\partial_{x_2}^2) U_2 + \partial_{x_2} P = 0 \quad \text{in} \ Z,\\
\partial_{x_1} U_1+\partial_{x_2} U_2=0\quad \text{in} \ Z,\\
U_1=0 \quad \text{on} \ J_{0}\cup J_{1}\\
U_2=0 \quad \text{on} \ J_{0},\\
-\partial_s^2 U_2 -\partial_{x_1}^2 U_2 = P -m_{\mathcal I}(P)\quad \text{on} \ J_{1}.
\end{array}
\right.
\end{equation}
In the above system, we write
$$
m_{\mathcal I }(P) := \frac{1}{2\pi} \int_0^{2\pi} P(x_1,1) \ dx_1
$$
and by using this notation in the last equation of \eqref{ns2.0.1}, we can replace the pressure that should satisfies a relation of the form
\eqref{meanpressure} by the pressure $P$ defined up to a function $c_P$ of $s$. 
In that way, we can, in what follows, impose another condition on $P$ (typically that its mean on an open set is zero).

To show \cref{SpectralInequality}, we first truncate $U$ and $P$ in a neighborhood of $\{s=s_{0}\}$, with 
$$
s_{0}:=\frac{S_{0}}{2}.
$$ 
We thus consider $\chi\in C^{\infty}_{0}((0,S_{0}))$, satisfying $0\leq \chi\leq 1$ and 
\begin{equation}\label{defchi}
\chi(s)=\begin{cases} 1 & \text{if } |s-s_{0}| \leq S_{0}/8, \\  0 & \text{if } |s-s_{0}| \geq S_{0}/6. 
\end{cases}
\end{equation}
We work with the following localized solutions 
\begin{equation}\label{defup}
u(s,x_1,x_2):=\chi(s) U(s,x_1,x_2), \quad p(s,x_1,x_2):=\chi(s) P(s,x_1,x_2)
\end{equation}
that satisfy 
\begin{equation}\label{ns2.1.0}
\left\{\begin{array}{l}
-\partial_{s}^2 u_1 - (\partial_{x_1}^2+\partial_{x_2}^2) u_1 + \partial_{x_1} p = f_{1} \quad \text{in} \ Z\\
-\partial_{s}^2 u_2 - (\partial_{x_1}^2+\partial_{x_2}^2) u_2 +  \partial_{x_2} p = f_{2} \quad \text{in} \ Z\\
\partial_{x_1} u_1+\partial_{x_2} u_2= 0\quad \text{in} \ Z\\
u_1=0 \quad \text{on} \ J_{0}\cup J_{1}\\
u_2=0 \quad \text{on} \ J_{0}\,\\
-\partial_s^2 u_2 -\partial_{x_1}^2 u_2 = f_{3}+ p-m_{\mathcal I}(p) \quad \text{on} \ J_{1},
\end{array}
\right.
\end{equation}
where 
\begin{equation}\label{deff123}
f_{1}:=-\chi'' U_1-2\chi' \partial_s U_1,\quad
f_{2}:=-\chi'' U_2-2\chi' \partial_s U_2,\quad
\quad
f_{3}:=-\chi'' (U_2)_{|_{J_{1}}}-2\chi'  (\partial_s U_2)_{|_{J_{1}}}.
\end{equation}
We also have that
\begin{equation}\label{ns2.1.1}
u=0 \quad \text{and} \quad p=0 \quad \text{if} \ s\notin \left[\frac{1}{3}S_0,\frac{2}{3}S_0\right].
\end{equation}
As usual, we can use the three first equations to obtain the following equation for the pressure:
\begin{equation}\label{ns2.2}
-\Delta p=- (\partial_{x_1}^2+\partial_{x_2}^{2}) p=\partial_{x_1} f_{1}+\partial_{x_2} f_{2}=0.
\end{equation}

\section{A global observability estimate on the pressure}\label{Sec_pressure}
In this section, we prove a global estimate on the pressure. We first introduce our weight and the corresponding conjugated operators.
We then state our main result, that is \cref{Estimateonpressureterms}. Then we show a first estimate on the pressure involving high frequency pressure terms at the boundary. Such terms are then estimated by showing some a priori estimates and this allows us to prove \cref{Estimateonpressureterms}.

\subsection{Choice of the weight and conjugated operators}\label{choice}
Let us consider a nonempty open set $\omega_0$ such that $\overline{\omega_0}\subset \omega$.
Let
$$
\lambda>0, \tau>0.
$$
Then we consider $\widetilde \psi\in C^{\infty}(\overline{\Omega};\mathbb{R}^+)$, such that 
\begin{equation}\label{weight2}
\widetilde{\psi}(x_1,x_2)=1-x_2 \text{ in a neighborhood of } \{x_2=1\},  \text{ and } \widetilde{\psi}(x_1,x_2)=x_2 \text{ in a neighborhood of } \{x_2=0\}
\end{equation}
and such that all its critical points belong to $\omega_0$:
\begin{equation} \label{weight41}
\nabla \widetilde{\psi} (x) =0 \implies x\in\omega_0.
\end{equation}
We set
\begin{equation}\label{weight1}
\varphi(s,x):=e^{\lambda \widetilde \psi(x)-(s-s_{0})^2}, \quad 
\varphi_{0}(s) :=e^{-(s-s_{0})^2}.
\end{equation}
Note that with our above choices, 
$$
\varphi_0(s)= \varphi(s,\cdot,0)=\varphi(s,\cdot,1)=\min_{x\in \overline{\Omega}} \varphi(s,x).
$$
We recall that we define $(u,p)$ from $(U,P)$ by \eqref{defup} (truncation in $s$)
and that the source $f_i$ are defined by \eqref{deff123}. We then define
\begin{equation}\label{weight3}
v:=e^{\tau \varphi} u, \quad q:=e^{\tau \varphi} p, \quad g_{i}:=e^{\tau \varphi} f_{i} \quad (i\in\{1,\dots,3\}).
\end{equation}
In order to take into account the dependence in $s$, we write
$$
z=(s,x)=(s,x_1,x_2)\in Z, \quad 
\nabla_{z} = \begin{bmatrix} \partial_{s}\\ \partial_{x_1}\\ \partial_{x_2}\end{bmatrix}, \quad
\Delta_{z}=\partial_{s}^2+\partial_{x_1}^2+\partial_{x_2}^2,
$$
and their tangential counterparts
$$
\nabla_{s,x_1} = \begin{bmatrix} \partial_{s}\\ \partial_{x_1}\end{bmatrix}, \quad
\Delta_{s,x_1}=\partial_{s}^2+\partial_{x_1}^2.
$$ 
We keep our previous notation 
$$
\nabla = \begin{bmatrix}  \partial_{x_1}\\ \partial_{x_2}\end{bmatrix}, \quad 
\Delta=\partial_{x_1}^2+\partial_{x_2}^2.
$$

The equations satisfied by $v$ and $q$ can be written with the introduction of the following conjugated operators:
\begin{align}
Q_{\varphi} & 
:= -e^{\tau\varphi}\Delta_{z}e^{-\tau\varphi} 
= -\Delta_{z} +2\tau\nabla_{z}\varphi\cdot\nabla_{z}-\tau^{2}|\nabla_{z}\varphi|^{2} +\tau(\Delta_{z} \varphi ),
\\
D_{\varphi} & 
:= - e^{\tau\varphi}\Delta e^{-\tau\varphi} 
= -\Delta+2\tau\left(\nabla \varphi \right)\cdot \nabla - \tau^{2}|\nabla\varphi|^{2} +\tau (\Delta\varphi ),
\label{defDphi}
\\
S_{\varphi} &
:= -e^{\tau\varphi_{0}} \Delta_{s,x_1}e^{-\tau\varphi_{0}} 
= -\Delta_{s,x_1}+2\tau(\partial_{s}\varphi_{0})\partial_{s}-\tau^{2}(\partial_{s}\varphi_{0})^{2}+\tau(\partial_{s}^{2}\varphi_{0}).
\end{align}
Then we deduce from \eqref{ns2.1.0} and \eqref{ns2.2} the following conjugated system:
\begin{equation}\label{ns2.1conj}
\left\{\begin{array}{l}
Q_{\varphi } v_1 + e^{\tau\varphi}\partial_{x_1}p = g_{1} \quad \text{in} \ Z,\\
Q_{\varphi} v_2 +  e^{\tau\varphi}\partial_{x_2}p = g_{2} \quad \text{in} \ Z,\\
D_{\varphi} q =0 \quad \text{in} \ Z,\\
v_1=0 \quad \text{on} \ J_{0}\cup J_{1},\\
v_2=0 \quad \text{on} \ J_{0},\\
S_{\varphi} v_2 = g_{3}+ q-m_{\mathcal I}(q) \quad \text{on} \ J_{1}.
\end{array}
\right.
\end{equation}
We define $h_i$ by
\begin{equation}\label{weight5}
h_{i}:=e^{\tau \varphi_{0}} f_{i} \quad (i=1,2,3).
\end{equation}
The main result of this section is the following result.
\begin{Theorem}\label{Estimateonpressureterms}
There exist $\lambda_0=\lambda_0(\widetilde \psi,S_0)>0$ and $\tau_0=\tau_0(\widetilde \psi,S_0)>0$ such
that for any $\lambda \geq \lambda_0$ and $\tau\geq \tau_0$, there exists $C=C(\lambda, \widetilde \psi,S_0)>0$
such that for any $\Lambda>0$ and for any $(a_{j})_{j}\in \mathbb C^{\mathbb N}$, the function $q$ defined by 
\eqref{DefofUandP}, \eqref{defup} and \eqref{weight3} satisfies 
\begin{multline}\label{estimatepressure0000}
\tau^{3}\| q \|_{L^{2}(Z)}^{2} + \tau \| \nabla  q \|^{2}_{L^{2}(Z)} 
+ \tau^{3} \left| q-m_{\mathcal I}(q)\right|_{L^{2}(J_{1})}^{2}
+ \tau \left| \nabla q \right|_{L^{2}((0,S_0)\times \partial \Omega)}^{2} 
\\
\leq C\left( \tau^{3} \| q \|^{2}_{L^{2}((0,S_0)\times \omega_0)} 
+ \tau \| \nabla q \|^{2}_{L^{2}((0,S_0)\times \omega_0)}   
+\tau\left(\left\| \partial_{x_1} h_1\right\|_{L^2(Z)}^2
+\left\| \partial_{x_1} h_2\right\|_{L^2(Z)}^2
+\left| \partial_{x_1} h_3\right|_{L^2(J_1)}^2\right)\right).
\end{multline}
\end{Theorem}
We prove this theorem in the remainder of this section.

\subsection{A first estimate on the pressure}
In order to prove \cref{Estimateonpressureterms}, we exploit that $q$ satisfies the third equation of \eqref{ns2.1conj},
where $D_{\varphi}$ is defined by \eqref{defDphi}.
Since we do not have any boundary condition, we need to split the boundary value of $q$ into high and low frequencies. More precisely, for $\mathcal{Q}\in H^2(\Omega)$,
we introduce the Fourier coefficients of the trace of $\mathcal{Q}:$
\begin{equation}\label{ak}
a_k(\mathcal{Q}):=
\begin{bmatrix}
\frac{1}{2\pi} \int_0^{2\pi} \mathcal{Q}(x_1,0) e^{-i k x_1} \ dx_1
\\
\frac{1}{2\pi} \int_0^{2\pi} \mathcal{Q}(x_1,1) e^{-i k x_1} \ dx_1 
\end{bmatrix} \quad (k\in \mathbb{Z}).
\end{equation}
We then define the sets of low tangential frequencies and high tangential frequencies:
\[
\mathrm{LF}_{\tau}:=\{ k\in\mathbb Z,\; k^{2}\leq \frac{\tau^{2}}{2}\inf |\partial_{x_2}\varphi |^{2 } \},
\quad 
\mathrm{HF}_{\tau}:=\{ k\in\mathbb Z,\; k^{2}> \frac{\tau^{2}}{2}\inf |\partial_{x_2}\varphi |^{2 } \}.
\]
In the above definition, the infimum of $\partial_{x_2}\varphi$ is taken for $x\in \partial \Omega$ and $s\in [0,S_0]$.

Due to \eqref{weight2} and  \eqref{weight1}, we have
$$
\inf |\partial_{x_2}\varphi |=\lambda e^{-s_0^2}.
$$
\begin{Proposition}\label{Carlemanonpressureprop}
There exist $\lambda_0=\lambda_0(\widetilde \psi,S_0)>0$ and $\tau_0=\tau_0(\widetilde \psi,S_0)>0$ such
that for any $\lambda \geq \lambda_0$ and $\tau\geq \tau_0$, there exists $C=C(\lambda, \widetilde \psi,S_0)>0$
such that for any $s\in [0,S_0],$ and any $\mathcal Q\in H^{2}(\Omega)$,
\begin{multline}\label{estimatepressure00}
\tau^{3}\| \mathcal Q \|_{L^{2}(\Omega)}^{2} + \tau \| \nabla \mathcal Q \|^{2}_{L^{2}(\Omega)} 
+ \tau | \partial_{x_2}\mathcal Q|_{L^{2}(\partial \Omega)}^{2} 
+\sum_{k\in \mathrm{LF}_{\tau}}\tau ( \tau^{2}+k^2) |a_k(\mathcal Q) |^{2}\\
\leq C\left( \|D_{\varphi}\mathcal Q \|_{L^{2}(\Omega)}^{2}+\tau^{3} \| \mathcal Q \|^{2}_{L^{2}(\omega_0)} 
+ \tau \| \nabla \mathcal Q \|^{2}_{L^{2}(\omega_0)}   
+\sum_{k\in \mathrm{HF}_{\tau}}\tau ( \tau^{2}+k^2) |a_k(\mathcal Q) |^{2}\right).
\end{multline}
\end{Proposition}
\begin{proof}
We can decompose the operator $D_{\varphi}$ (see \eqref{defDphi}) as follows
$D_{\varphi} = \mathcal S+\mathcal A+ \mathcal R$, where
$$
\mathcal S = -\Delta -\tau^{2}|\nabla\varphi|^{2},
\quad
\mathcal A = 2\tau\nabla\varphi\cdot\nabla +2\tau (\Delta \varphi ),
\quad
\mathcal R  =-\tau (\Delta \varphi ).
$$
Then, after some standard computation, we can obtain that
\begin{multline}\label{Lu23:07}
\int_\Omega (\mathcal S \mathcal Q) (\mathcal A \mathcal Q) dx 
=  
\tau \int_{\Omega}\left( 2\nabla^{2}\varphi(\nabla \mathcal Q,\nabla\mathcal Q) + \Delta\varphi |\nabla\mathcal Q|^{2}\right)  dx
+\tau^{3}\int_{\Omega}\left(  2\nabla^{2}\varphi( \nabla\varphi,\nabla\varphi)  - |\nabla\varphi|^{2} \Delta\varphi\right)|\mathcal Q|^{2}  dx \\
	-\tau\int_{\Omega}(\Delta^{2}\varphi) |\mathcal Q|^{2} dx
+\mathcal B,
\end{multline}
where 
$$
\nabla^2 \varphi=\left(\frac{\partial^2 \varphi}{\partial x_i\partial x_j}\right)_{i,j}
$$
and where
$\mathcal B$ corresponds to the boundary terms:
\begin{multline*}
\mathcal B=-2\tau \int_{\partial\Omega} \partial_{n}\mathcal Q (\nabla\varphi\cdot\nabla \mathcal Q) d\Gamma
	+\tau\int_{\partial\Omega} (\partial_{n}\varphi) | \nabla \mathcal Q|^{2}  d\Gamma
	-2\tau \int_{\partial\Omega} (\Delta\varphi)(\partial_{n}\mathcal Q)\mathcal Q d\Gamma 
\\ 
+\tau\int_{\partial\Omega} (\partial_{n}\Delta\varphi) \mathcal Q^{2}d\Gamma
-\tau^{3}\int_{\partial\Omega} |\nabla\varphi|^{2}(\partial_{n}\varphi) \mathcal Q^{2} d\Gamma.
\end{multline*}
Using \eqref{weight2}, we have that $\partial_{n}\varphi<0$ and we can simplify the above quantity:
\begin{multline*}
\mathcal B=\tau \int_{\partial\Omega} \left|\partial_{x_2}\varphi\right| \left(\partial_{x_2}\mathcal Q\right)^2  d\Gamma
	-\tau \int_{\partial\Omega} \left|\partial_{x_2}\varphi\right| \left(\partial_{x_1}\mathcal Q\right)^2  d\Gamma
	-2\tau \int_{\partial\Omega} (\partial_{x_2}^2\varphi)(n\cdot e_2)(\partial_{x_2}\mathcal Q)\mathcal Q d\Gamma 
\\ 
+\tau\int_{\partial\Omega} (\partial_{x_2}^3\varphi) (n\cdot e_2) \mathcal Q^{2}d\Gamma
+\tau^{3}\int_{\partial\Omega} \left|\partial_{x_2}\varphi\right|^3 \mathcal Q^{2} d\Gamma.
\end{multline*}
Combining the above relation with \eqref{weight1}, there exists $\tau_1=\tau_1(S_0)>0$ such that for any $\tau\geq \tau_1$, we have
\begin{multline}\label{Lu23:06}
\mathcal B\geq \frac12 \tau \lambda \varphi_0 \int_{\partial\Omega}  \left(\partial_{x_2}\mathcal Q\right)^2  d\Gamma
	-\tau \lambda \varphi_0  \int_{\partial\Omega} \left(\partial_{x_1}\mathcal Q\right)^2  d\Gamma
	+\frac34 \tau^{3}\lambda^3 \varphi_0^3 \int_{\partial\Omega}  \mathcal Q^{2} d\Gamma
	\\
	\geq \frac12 \tau \lambda \varphi_0 \int_{\partial\Omega}  \left(\partial_{x_2}\mathcal Q\right)^2  d\Gamma
	+2\pi \tau \lambda \varphi_0  \sum_{k\in \mathbb{Z}} \left(\frac34 \tau^{2}\lambda^2 \varphi_0^2 - k^2\right)|a_k(\mathcal{Q})|^2.
\end{multline}
Using \eqref{weight1} and \eqref{weight41}, there exist $C_1=C_1(\widetilde \psi)$, $C_2=C_2(\widetilde \psi)$, $\tau_2=\tau_2(S_0,\widetilde \psi)$ and 
$\lambda_1=\lambda_1(\widetilde \psi)$ such that for $\lambda \geq \lambda_1$
and $\tau\geq \tau_2$, 
\begin{multline}\label{21:40}
\tau \int_{\Omega}\left( 2\nabla^{2}\varphi(\nabla \mathcal Q,\nabla\mathcal Q) + \Delta\varphi |\nabla\mathcal Q|^{2}\right)  dx
+\tau^{3}\int_{\Omega}\left(  2\nabla^{2}\varphi( \nabla\varphi,\nabla\varphi)  - |\nabla\varphi|^{2} \Delta\varphi\right)|\mathcal Q|^{2}  dx
	-\tau\int_{\Omega}(\Delta^{2}\varphi) |\mathcal Q|^{2} dx
\\
\geq 
C_1\int_{\Omega} \left(\tau \lambda^2 \varphi \left| \nabla\mathcal Q \right|^2+\tau^3 \lambda^4 \varphi^3 \left| \mathcal Q \right|^2 \right) dx
-C_2\int_{\omega_0} \left(\tau \lambda^2 \varphi \left| \nabla\mathcal Q \right|^2+\tau^3 \lambda^4 \varphi^3 \left| \mathcal Q \right|^2 \right) dx.
\end{multline}

Finally, combining \eqref{Lu23:07}, \eqref{Lu23:06} and \eqref{21:40}, we deduce the existence of $\lambda_0>0$ and $\tau_0>0$ such
that for any $\lambda \geq \lambda_0$ and $\tau\geq \tau_0$, there exist $C_3=C_3(\lambda, \widetilde \psi,S_0)>0$
and $C_4=C_4(\lambda, \widetilde \psi,S_0)>0$ such that
\begin{multline*}
\|D_{\varphi}\mathcal Q \|_{L^{2}(\Omega)}^{2}  \geq  \frac 12 \| (\mathcal S +\mathcal A)\mathcal Q \|^{2}_{L^{2}(\Omega)} - \|\mathcal R \mathcal Q\|^{2}_{L^{2}(\Omega)} \geq \Re  \left( \mathcal S\mathcal Q,\mathcal A \mathcal Q \right)_{L^{2}(\Omega)}- \|\mathcal R \mathcal Q\|^{2}_{L^{2}(0,1)}\\ 
 \geq
C_3\left( \int_{\Omega} \left(\tau |\nabla \mathcal Q|^{2}+ \tau^{3}|\mathcal Q|^{2}\right) \ dx
+\tau \int_{\partial\Omega}  \left(\partial_{x_2}\mathcal Q\right)^2  d\Gamma
+ \tau \sum_{k\in \mathrm{LF}_{\tau}} (\tau^2+k^2) |a_k(\mathcal{Q})|^2\right)
\\
-C_4\left( \int_{\omega_0} \left(\tau |\nabla \mathcal Q|^{2}+ \tau^{3}|\mathcal Q|^{2}\right) \ dx
+\tau \sum_{k\in \mathrm{HF}_{\tau}}  k^2 |a_k(\mathcal{Q})|^2\right).
\end{multline*}
\end{proof}

Applying the above result to $\mathcal Q=q$ solution of \eqref{ns2.1conj} and integrating into $(0,S_0)$, 
we deduce that
\begin{multline}\label{estimatepressure0}
\tau^{3}\| q \|_{L^{2}(Z)}^{2} + \tau \| \nabla  q \|^{2}_{L^{2}(Z)} 
+ \tau | \partial_{x_2} q|_{L^{2}((0,S_0)\times \partial \Omega)}^{2} 
+\sum_{k\in \mathrm{LF}_{\tau}}\tau ( \tau^{2}+k^2) |a_k(q) |_{L^2(0,S_0)}^{2}\\
\leq C\left( \tau^{3} \| q \|^{2}_{L^{2}((0,S_0)\times \omega_0)} 
+ \tau \| \nabla q \|^{2}_{L^{2}((0,S_0)\times \omega_0)}   
+\sum_{k\in \mathrm{HF}_{\tau}}\tau ( \tau^{2}+k^2) |a_k(q) |_{L^2(0,S_0)}^{2}\right).
\end{multline}
We recall that $q(s,\cdot)\equiv 0$ if $|s-s_{0}| \geq S_{0}/6$ due to the support of $\chi$ (see \eqref{defchi}).
 Next, we will estimate the high tangential frequencies of the pressure.

\subsection{Estimates in the high frequency regime} 
We define
$$
Y:=(0,S_{0})\times (0,1), \quad 
I_{i}:=(0,S_{0})\times\{i\}, \ i=0,1.
$$
We recall that $(u,p)$ is defined by \eqref{defup} and $(f_1,f_2,f_3)$ is defined by \eqref{deff123}.
We define $(u^k,p^k,f_1^k,f_2^k,f_3^k)$ the Fourier coefficients of $(u,p,f_1,f_2,f_3)$ in the $x_1$ direction. For instance
$$
u^k(s,x_2):=\frac{1}{2\pi} \int_0^{2\pi} u(s,x_1,x_2) e^{-i k x_1} \ dx_1 \quad ((s,x_2)\in Y).
$$
Finally, with $\tau>0$ and $\varphi_0$ defined by \eqref{weight1}, we set
\begin{equation}\label{weight4}
w^k=e^{\tau \varphi_{0}} u^k, \quad \pi^k=e^{\tau \varphi_{0}} p^k, \quad h_{i}^k=e^{\tau \varphi_{0}} f_{i}^k \quad (i=1,2,3).
\end{equation}
Note that $h_i^k$ are the Fourier coefficients of the functions $h_i$ defined by \eqref{weight5}.
Since  $\varphi_{0}$ only depends on $s$, and using \eqref{weight1}, \eqref{weight3}, we have
\begin{equation}\label{ns6.6}
a_k(q)=
\begin{bmatrix}
(\pi^k)_{|I_0}
\\
(\pi^k)_{|I_1}
\end{bmatrix}.
\end{equation}
Let us define the following conjugated operators:
\begin{align*}
Q_{k,\varphi_{0}} & := e^{\tau\varphi_{0}}(-\partial_{s}^{2}-\partial_{x_2}^{2}+k^{2})e^{-\tau\varphi_{0}} 
= -\partial_{s}^{2}-\partial_{x_2}^{2}+k^{2} +2\tau \varphi_{0}'\partial_{s}
	-\tau^{2}(\varphi_{0}')^{2} +\tau \varphi_{0}'',
	\\
S_{k,\varphi_{0}} &:= e^{\tau\varphi_{0}} ( -\partial_{s}^{2}+k^{2})e^{-\tau\varphi_{0}}
= -\partial_{s}^{2}+k^{2} +2\tau \varphi_{0}'\partial_{s}
	-\tau^{2}(\varphi_{0}')^{2} +\tau \varphi_{0}''.
\end{align*}
Then, for $k\in \mathbb{Z}$, \eqref{ns2.1.0} transforms into
\begin{equation}\label{ns3.1conj}
\left\{\begin{array}{l}
Q_{k,\varphi_{0} } w_1^k + ik \pi^k = h_{1}^{k} \quad \text{in} \ Y,\\
Q_{k,\varphi_{0}} w_2^k +  \partial_{x_2} \pi^k = h_{2}^{k} \quad \text{in} \ Y,\\
\div_{k} w^{k}= 0 \quad \text{in} \ Y,\\
w_1^k=0 \quad \text{on} \ I_{0}\cup I_{1},\\
w_2^k=0 \quad \text{on} \ I_{0},\\
S_{k,\varphi} w_2^k = h_{3}^{k}+ \pi^k \quad \text{on} \ I_{1},
\end{array}
\right.
\end{equation}
where 
$$
\div_{k} \begin{bmatrix}
f_1\\ f_2
\end{bmatrix}
=ik f_1+\partial_{x_2} f_2.
$$
The relation \eqref{ns2.1.1} yields
\begin{equation}\label{ns2.1.1k}
w^k=0 \quad \text{and} \quad \pi^k=0 \quad \text{if} \ s\notin \left[\frac{1}{3}S_0,\frac{2}{3}S_0\right].
\end{equation}

\begin{Proposition}\label{PropositionHF}
If the solution of \eqref{ns3.1conj} satisfies \eqref{ns2.1.1k}, then there 
exist $\lambda_3=\lambda_3(S_0)>0$ and $C(S_0)>0$ such that for any $\lambda\geq \lambda_3$ and $k\in \mathrm{HF}_{\tau}$,
$$
 \| \pi^{k} \|_{L^{2}(I_0\cup I_1)} \leq \frac{C}{(k^2+\tau^2)^{1/4}} \left(
\left\|h_1^{k}\right\|_{L^2(Y)}
+\left\|h_2^{k}\right\|_{L^2(Y)}
+\left|h_3^{k}\right|_{L^2(I_1)}\right).
$$
\end{Proposition}
\begin{proof}
We multiply the first line \eqref{ns3.1conj} by $w_1^{k}$, the second line \eqref{ns3.1conj} by $w_2^{k}$, and the last line \eqref{ns3.1conj} by $w_2^{k}$.
Integrating by parts and summing up yield
\begin{multline}\label{ns5.0}
\int_{Y} |\partial_{s}w^{k}|^{2} \ dy
+ \int_{Y} |\partial_{x_2}w^{k} |^{2} \ dy
+k^{2}\int_{Y}|w^{k}|^{2} \ dy 
- \tau^{2}\int_{Y}(\varphi_{0}')^{2} |w^{k}|^{2} \ dy 
\\
+  \int_{I_{1}} |\partial_{s}w_{2}^{k}|^{2} \ ds
+  k^{2}\int_{I_{1}} |w_{2}^{k}|^{2} \ ds
- \tau^{2}\int_{I_{1}} (\varphi_{0}')^{2} |w_{2}^{k}|^{2} \ ds
\\
= \Re \int_{Y} h_1^{k}\overline{w_1^{k}} \ dy
+\Re \int_{Y} h_2^{k}\overline{w_2^{k}} \ dy
+\Re \int_{I_{1}} h_{3}^{k} \overline{w^{k}_{2}} \ ds.
\end{multline}
Now, since $k\in \mathrm{HF}_{\tau}$, we have
$$
k^{2}> \frac{\tau^{2}}{2}\inf |\partial_{x_2}\varphi |^{2 } =\frac{\tau^{2}}{2} \lambda^2 e^{-2 s_0^2}.
$$
On the other hand, 
$$
\sup_{[0,S_0]} |\varphi_0'| \leq S_0.
$$
From the two previous relations, we deduce the existence of $\lambda_3=\lambda_3(S_0)>0$ such that for $\lambda\geq \lambda_3$
and for $k\in \mathrm{HF}_{\tau}$
\[
k^{2}-\tau^{2}|\varphi_{0}' |^{2} \geq \frac{1}{2}(\tau^{2}+k^{2} ).
\]
Combining the above relation and \eqref{ns5.0} yields
\begin{multline}\label{ns5.1}
(\tau^{2}+k^{2} )\left\|\begin{bmatrix}
\partial_{s}w^{k}, \partial_{x_2}w^{k}
\end{bmatrix}\right\|_{L^2(Y)}^{2}
+(\tau^{2}+k^{2} )^2 \left\|w^{k}\right\|_{L^2(Y)}^{2}
+(\tau^{2}+k^{2} )\left|\partial_{s}w_{2}^{k}\right|_{L^2(I_1)}^{2}
+(\tau^{2}+k^{2} )^2 \left|w_2^{k}\right|_{L^2(I_1)}^{2}
\\
\lesssim
\left\|h_1^{k}\right\|_{L^2(Y)}^{2}
+\left\|h_2^{k}\right\|_{L^2(Y)}^{2}
+\left|h_3^{k}\right|_{L^2(I_1)}^{2}.
\end{multline}

Now, we write \eqref{ns3.1conj} under the form
\begin{equation}\label{ns3.1conj2}
\left\{\begin{array}{l}
\left(-\partial_{s}^{2}-\partial_{x_2}^{2}+k^{2}\right) w_1^k + ik \pi^k = H_{1}^{k} \quad \text{in} \ Y,\\
\left(-\partial_{s}^{2}-\partial_{x_2}^{2}+k^{2}\right) w_2^k +  \partial_{x_2} \pi^k = H_{2}^{k} \quad \text{in} \ Y,\\
\div_{k} w^{k}= 0 \quad \text{in} \ Y,\\
w_1^k=0 \quad \text{on} \ I_{0}\cup I_{1},\\
w_2^k=0 \quad \text{on} \ I_{0},\\
\left( -\partial_{s}^{2}+k^{2}\right) w_2^k = H_{3}^{k}+ \pi^k \quad \text{on} \ I_{1},
\end{array}
\right.
\end{equation}
where
\begin{align*}
H_1^k& := -2\tau \varphi_{0}'\partial_{s}w_1^k +\tau^{2}(\varphi_{0}')^{2}w_1^k  -\tau \varphi_{0}''w_1^k+h_{1}^{k},
	\\
H_2^k& := -2\tau \varphi_{0}'\partial_{s}w_2^k +\tau^{2}(\varphi_{0}')^{2}w_2^k  -\tau \varphi_{0}''w_2^k+h_{2}^{k},
	\\
H_3^k&:= -2\tau \varphi_{0}'\partial_{s}w_2^k +\tau^{2}(\varphi_{0}')^{2}w_2^k  -\tau \varphi_{0}''w_2^k++h_{3}^{k}.
\end{align*}
From \eqref{ns5.1}, we deduce
\begin{equation}\label{ns5.2}
\left\|H_1^{k}\right\|_{L^2(Y)}
+\left\|H_2^{k}\right\|_{L^2(Y)}
+\left|H_3^{k}\right|_{L^2(I_1)}
\lesssim
\left\|h_1^{k}\right\|_{L^2(Y)}
+\left\|h_2^{k}\right\|_{L^2(Y)}
+\left|h_3^{k}\right|_{L^2(I_1)}.
\end{equation}
We multiply the first line \eqref{ns3.1conj2} by $-\partial_s^2 w_1^{k}$, the second line \eqref{ns3.1conj2} by $-\partial_s^2 w_2^{k}$, and the last line \eqref{ns3.1conj2} by $-\partial_s^2 w_2^{k}$.
Integrating by parts, summing up and using \eqref{ns5.2} yield
\begin{multline}\label{ns5.3}
\left\|\partial_{s}^2 w^{k}\right\|_{L^2(Y)}^{2}
+\left\|\partial_s\partial_{x_2} w^{k}\right\|_{L^2(Y)}^{2}
+k^{2} \left\| \partial_sw^{k}\right\|_{L^2(Y)}^{2}
+\left|\partial_{s}^2 w_{2}^{k}\right|_{L^2(I_1)}^{2}
+k^{2} \left| \partial_{s}w_2^{k}\right|_{L^2(I_1)}^{2}
\\
\lesssim
\left\|h_1^{k}\right\|_{L^2(Y)}^{2}
+\left\|h_2^{k}\right\|_{L^2(Y)}^{2}
+\left|h_3^{k}\right|_{L^2(I_1)}^{2}.
\end{multline}
Next, we write \eqref{ns3.1conj2} under the form
\begin{equation}\label{ns3.3conj}
\left\{\begin{array}{l}
\left(-\partial_{x_2}^{2}+k^{2}\right) w_1^k + ik \pi^k = \widetilde H_{1}^{k} \quad \text{in} \ Y,\\
\left(-\partial_{x_2}^{2}+k^{2}\right) w_2^k +  \partial_{x_2} \pi^k = \widetilde H_{2}^{k} \quad \text{in} \ Y,\\
\div_{k} w^{k}= 0 \quad \text{in} \ Y,\\
w_1^k=0 \quad \text{on} \ I_{0}\cup I_{1},\\
w_2^k=0 \quad \text{on} \ I_{0},\\
k^{2} w_2^k = \widetilde H_{3}^{k}+ \pi^k \quad \text{on} \ I_{1},
\end{array}
\right.
\end{equation}
where
$$
\widetilde H_1^k:= H_1^k+\partial_{s}^{2}w_1^k, \quad
\widetilde H_2^k:= H_2^k+\partial_{s}^{2}w_2^k, \quad
\widetilde H_3^k:= H_3^k+\partial_{s}^{2}w_1^k.
$$
From \eqref{ns5.2} and \eqref{ns5.3}, we deduce
\begin{equation}\label{ns6.0}
\left\|\widetilde H_1^{k}\right\|_{L^2(Y)}
+\left\|\widetilde H_2^{k}\right\|_{L^2(Y)}
+\left|\widetilde H_3^{k}\right|_{L^2(I_1)}
\lesssim
\left\|h_1^{k}\right\|_{L^2(Y)}
+\left\|h_2^{k}\right\|_{L^2(Y)}
+\left|h_3^{k}\right|_{L^2(I_1)}.
\end{equation}

The first two lines of \eqref{ns3.3conj} can be written as
\begin{equation}\label{ns6.1}
\nabla \left(e^{ikx_1} \pi^k\right)= e^{ikx_1} \begin{bmatrix}
\widetilde H_{1}^{k}\\\widetilde H_{2}^{k}
\end{bmatrix}+\Delta \left(e^{ikx_1} w^k\right)
\end{equation}
and thus
\begin{equation}\label{ns6.2}
\left\| \nabla \left(e^{ikx_1} \pi^k\right) \right\|_{H^{-1}(\Omega)}\leq 
	\left\| e^{ikx_1} \begin{bmatrix}\widetilde H_{1}^{k}\\\widetilde H_{2}^{k}\end{bmatrix} 
	\right\|_{H^{-1}(\Omega)}
	+
	\left\| \nabla \left(e^{ikx_1} w^k\right) \right\|_{L^{2}(\Omega)} 
\end{equation}
and using that $k\neq 0$, we deduce from the above estimate that
\begin{equation}\label{ns6.3}
\left\|  e^{ikx_1} \pi^k\right\|_{H^{-1}(\Omega)}\leq 
	\left\| e^{ikx_1} \begin{bmatrix}\widetilde H_{1}^{k}\\\widetilde H_{2}^{k}\end{bmatrix} 
	\right\|_{H^{-1}(\Omega)}
	+
	\left\| \nabla \left(e^{ikx_1} w^k\right) \right\|_{L^{2}(\Omega)}.
\end{equation}
Combining \eqref{ns6.2} and \eqref{ns6.3} with the Ne\v{c}as inequality (see, for instance, \cite[p.231, Theorem IV.1.1]{BoyerFabrie}), 
we deduce that
\begin{equation}
k\left\|  e^{ikx_1} \pi^k\right\|_{L^2(\Omega)}\lesssim
	\left\|\widetilde H_1^{k}\right\|_{L^2(0,1)}
	+\left\|\widetilde H_2^{k}\right\|_{L^2(0,1)}
	+k\left\| \nabla \left(e^{ikx_1} w^k\right) \right\|_{L^{2}(\Omega)}
\end{equation}
and thus, with \eqref{ns6.0} and \eqref{ns5.1}, 
\begin{equation}\label{ns6.4}
k\left\|  \pi^k\right\|_{L^2(Y)}\lesssim
\left\|h_1^{k}\right\|_{L^2(Y)}
+\left\|h_2^{k}\right\|_{L^2(Y)}
+\left|h_3^{k}\right|_{L^2(I_1)}.
\end{equation}

On the other hand, differentiating the divergence equation of system \eqref{ns3.3conj} with respect to $x_2$ and using \eqref{ns5.1},  \eqref{ns6.0}  yield
\[
\|\partial_{x_2}^{2} w^{k}_{2} \|_{L^{2}(Y)}  + \|\partial_{x_2} \pi^{k} \|_{L^{2}(Y)} \lesssim
\left\|h_1^{k}\right\|_{L^2(Y)}
+\left\|h_2^{k}\right\|_{L^2(Y)}
+\left|h_3^{k}\right|_{L^2(I_1)}.
\]
Then, combining the above relation with \eqref{ns6.4} and with a trace inequality, we deduce 
$$
 \| \pi^{k} \|_{L^{2}(I_0\cup I_1)} \lesssim\frac{1}{k^{1/2}} \left(
\left\|h_1^{k}\right\|_{L^2(Y)}
+\left\|h_2^{k}\right\|_{L^2(Y)}
+\left|h_3^{k}\right|_{L^2(I_1)}\right).
$$
Using that $k\in \mathrm{HF}_{\tau}$, we deduce the result.
\end{proof}
\subsection{Proof of \cref{Estimateonpressureterms}}
From \cref{PropositionHF} and from \eqref{ns6.6}, for $k\in \mathrm{HF}_{\tau}$
$$
( \tau^{2}+k^2) \left| a^{k}(q) \right|_{L^{2}(0,S_0)}^2 \lesssim
k^2 \left(
\left\|h_1^{k}\right\|_{L^2(Y)}^2
+\left\|h_2^{k}\right\|_{L^2(Y)}^2
+\left|h_3^{k}\right|_{L^2(I_1)}^2\right)
$$
and thus, with the Parceval formula, 
$$
\sum_{k\in \mathrm{HF}_{\tau}} ( \tau^{2}+k^2) |a_k(q) |_{L^2(0,S_0)}^{2}
\lesssim 
\left\| \partial_{x_1} h_1\right\|_{L^2(Z)}^2
+\left\| \partial_{x_1} h_2\right\|_{L^2(Z)}^2
+\left| \partial_{x_1} h_3\right|_{L^2(J_1)}^2,
$$
where we have used \eqref{weight4} and \eqref{weight5}.

Combining this estimate with \eqref{estimatepressure0} we finally obtain the sought result.
\begin{flushright}
$\qed$
\end{flushright}

\section{Proof of the spectral inequality}\label{Sec_spectral}
The proof of the spectral inequality, that is \eqref{ns4.1} is based on 
interpolation estimates. More precisely, it will be a consequence of \cref{globalinterpolation} stated below.
In order to show such a result, we first recall some interpolation inequalities available in the literature and then we combine them with the global pressure estimates, that is \cref{Estimateonpressureterms} to show \cref{globalinterpolation}. The last part of this section is devoted to the proof of the spectral inequality from the interpolation inequality.

First, we need the following notation for this section:
\begin{equation}\label{calO}
\mathcal O_0:=\left(s_0- \frac{S_{0}}{6},s_0+ \frac{S_{0}}{6}\right) \times\omega_0,
\quad
\mathcal O:=\left(s_0- \frac{S_{0}}{5},s_0+ \frac{S_{0}}{5}\right) \times\omega,
\end{equation}
\begin{equation}\label{Ztilde}
\widetilde{Z}:=\left(s_0- \frac{S_{0}}{10},s_0+ \frac{S_{0}}{10}\right) \times\Omega,
\quad
\widetilde{J}_1:=\left(s_0- \frac{S_{0}}{10},s_0+ \frac{S_{0}}{10}\right) \times \Gamma_1,
\end{equation}
\begin{equation}\label{Zhat}
\widehat{Z}:=\left(s_0- \frac{S_{0}}{9},s_0+ \frac{S_{0}}{9}\right) \times\Omega,
\quad
\widehat{J}_1:=\left(s_0- \frac{S_{0}}{9},s_0+ \frac{S_{0}}{9}\right) \times \Gamma_1.
\end{equation}
Note that $\widetilde{Z} \subset \widehat{Z}$ and $\widetilde{J}_1 \subset \widehat{J}_1$.

\subsection{Estimates on the velocity}
The two components of the velocity satisfy different boundary conditions (see \eqref{ns2.0.1}).
We start by an estimate on the first component $U_{1}$ that satisfies homogeneous Dirichlet boundary conditions.
For the proof of this result, we refer to relation $(1)$ in Section 3 of \cite{LebeauRobbiano}.
\begin{Theorem}\label{Carlemansurv1}
There exist $C_{1}>0$ and $\mu_{1}\in(0,1)$ such that 
for all $w\in H^{2}(Z)$ such that $w_{|_{J_{0}\cup J_{1}}}=0$,
\begin{equation}\label{Interpolationsurv1}
\left\| w \right\|_{H^{1}\left(\widetilde Z \right)} 
	\leq C_{1} \left\| w \right\|_{H^{1}\left(\widehat Z\right)}^{1-\mu_{1}} 
		\left( \left\| \Delta_{z} w \right\|_{L^{2}\left(\widehat Z\right)} + \left\| w \right\|_{L^{2}\left(\mathcal O_0\right)} 
		\right)^{\mu_{1}}.
\end{equation}
\end{Theorem}
Note that Theorem \ref{Carlemansurv1} is stated with an $H^{1}$ observation in \cite{LebeauRobbiano}, but we can transform it into an $L^{2}$ observation as in \eqref{Interpolationsurv1} by using a cut-off function and integrations by parts.

For the estimate of $U_2$, we note that it satisfies a Ventcel boundary condition on $J_{1}$ and the Dirichlet boundary condition on $J_{0}$. 
Hence, we use the following result, which is basically a consequence of a Carleman estimate obtained in \cite{Buffe17}. However for the sake of completeness, we 
prove the next result in \cref{Carlemansurv2proof}.
\begin{Theorem}\label{Carlemansurv2}
There exist $C_{2}>0$ and $\mu_{2}\in(0,1)$ such that 
\begin{multline}\label{Interpolationsurv2}
\left\| w \right\|_{H^{1}\left(\widetilde Z\right)}+\left| w \right|_{H^{1}\left(\widetilde J_1\right)}
 \leq C_{2} \left( \left\| w \right\|_{H^{1}\left(\widehat Z\right)}+\left| w \right|_{H^{1}\left(\widehat J_1\right)}
		 \right)^{1-\mu_{2}} 
\\
\times \left( \left\| \Delta_{z} w \right\|_{L^{2}\left(\widehat Z\right)} 
 	+ \left| \left(\partial_{x_{2}}w\right)_{|_{J_{1}}}-\Delta_{x_1,s} w_{|_{J_{1}}} \right|_{L^{2}(\widehat J_{1})}+ \left\| w \right\|_{L^{2}(\mathcal O_0)} 
\right)^{\mu_{2}}.
\end{multline}
for all $w\in H^{2}(Z)$ such that $w_{|_{J_{0}}}=0$ and $w_{|_{J_{1}}}\in H^{2}(J_{1}).$
\end{Theorem}
Note that both \cref{Carlemansurv1} and \cref{Carlemansurv2} hold for $\mu_3\in (0,1)$ such that
$\mu_3\leq \mu_1$ and $\mu_3\leq \mu_2$ (with a modification of the constants $C_1$ and $C_2$).
Thus, with $\mu_{3}=\min(\mu_{1},\mu_{2})$ and an adequate constant $C_3$, we can apply 
\cref{Carlemansurv1} with $w=U_{1}$ and \cref{Carlemansurv2} with $w=U_{2}$, where $U$ satisfies \eqref{ns2.0.1} and we deduce 
\begin{multline}\label{estimatevelocity}
\left\| U \right\|_{H^{1}\left(\widetilde Z\right)}+\left| U_2 \right|_{H^{1}\left(\widetilde J_1\right)}
 \leq C_{3} \left( \left\| U \right\|_{H^{1}\left(\widehat Z\right)}+\left| U_2 \right|_{H^{1}\left(\widehat J_1\right)}
		 \right)^{1-\mu_{3}} 
\\
\times \left( \left\| \nabla P \right\|_{L^{2}\left(\widehat Z\right)} 
 	+ \left| P_{|J_{1}}-m_{\mathcal I}(P)  \right|_{L^{2}(\widehat J_{1})}+ \left\| U \right\|_{L^{2}(\mathcal O_0)} 
\right)^{\mu_{3}}.
\end{multline}
We have used here the fact that on $J_1$,
$$
\partial_{x_2} U_2 = - \partial_{x_1} U_1=0.
$$
We are going now to combine  the above estimate with the estimates of the pressure terms obtained in \cref{Sec_pressure}.

\subsection{Patching the estimates together}\label{Patchingestimatesubsection}
Combining the previous estimates, we can now prove the following result.
\begin{Theorem}\label{globalinterpolation}
There exist $C>0$ and $\mu\in (0,1)$ such that for any $\Lambda>0$ and for any $(a_{j})_{j}\in \mathbb C^{\mathbb N}$, the function $U$ 
defined by \eqref{DefofUandP} satisfies 
$$
\left\| U \right\|_{H^{1}\left(\widetilde Z\right)}+\left| U_2 \right|_{H^{1}\left(\widetilde J_1\right)}
 \leq C \left( \left\| U \right\|_{H^{2}\left(Z\right)}+\left| U_2 \right|_{H^{2}\left(J_1\right)}\right)^{1-\mu}
\left\| U \right\|_{H^{2}(\mathcal O_0)}^\mu.
$$
\end{Theorem}
%
%
\begin{proof}
We start with the estimate \eqref{estimatepressure0000}, where we recall that
$q$ is given by \eqref{defup}, \eqref{weight3} and $h_i$, $i=1,2,3$ by \eqref{weight5}:
\begin{multline}\label{est4.3.1}
\tau^{3}\| e^{\tau\varphi}\chi P \|_{L^{2}(Z)}^{2}  + \tau \| e^{\tau\varphi}\chi\nabla P \|^{2}_{L^{2}(Z)} + \tau^{3} | e^{\tau\varphi_{0}}\chi (P-m_{\mathcal I}(P))|_{L^{2}(J_{1})}^{2} + \tau |e^{\tau\varphi_{0}} \chi\partial_{x_1}P|_{L^{2}(J_{1})}^{2} 
+ \tau |e^{\tau\varphi_{0}}\chi \partial_{x_2}P|_{L^{2}(J_{1})}^{2} 
\\
\leq C\Big( \tau^{3} \| e^{\tau\varphi} P \|^{2}_{L^{2}(\mathcal O_0)} 
+\tau \| e^{\tau\varphi}  \nabla P \|^{2}_{L^{2}(\mathcal O_0)}  
+\tau\left\| e^{\tau\varphi_{0}} \partial_{x_1} f_1\right\|_{L^2(Z)}^2
\\
+\tau\left\| e^{\tau\varphi_{0}}\partial_{x_1} f_2\right\|_{L^2(Z)}^2
+\tau\left| e^{\tau\varphi_{0}} \partial_{x_1} f_3\right|_{L^2(J_1)}^2\Big).
\end{multline}
Note that $f_1,f_2$ (respectively $f_3$) are supported in 
$$
Z_{8}:=\left( \supp \chi' \right)\times \mathcal{I} \times (0,1) \quad (\text{respectively in} \ J_8:=\left( \supp \chi' \right)\times \mathcal{I} \times \{1\}),
$$
and since $\supp \chi'\subset \left[s_0-S_0/6,s_0+S_0/6\right]$
\begin{equation}\label{supofweight}
\sup_{J_{8}}\varphi_{0} =\sup_{Z_{8}}\varphi_{0} \leq e^{-\frac{S_{0}^{2}}{36}}.
\end{equation}
Hence, from \eqref{deff123}
\begin{align}\label{est4.3.2}
\left\| e^{\tau\varphi_{0}} \partial_{x_1} f_1\right\|_{L^2(Z)}
+\left\| e^{\tau\varphi_{0}}\partial_{x_1} f_2\right\|_{L^2(Z)}
+\left| e^{\tau\varphi_{0}} \partial_{x_1} f_3\right|_{L^2(J_1)}
\leq C e^{\tau e^{-\frac{S_{0}^{2}}{36}}} \left(\| U \|_{H^{2}(Z)} + | U_2 |_{H^{2}(J_1)}\right).
\end{align}
In $\widehat Z$ (respectively in $\widehat J_{1}$), $\chi(s)\equiv1$, and 
\begin{equation}\label{infofweight}
\inf_{\widehat Z} \varphi = \inf_{\widehat Z} \varphi_{0}= \inf_{\widehat J_{1}} \varphi_{0} = e^{- \frac{S_{0}^{2}}{81}}.
\end{equation}
Combining \eqref{est4.3.1}, \eqref{est4.3.2} and \eqref{infofweight}, there exist $\tau_4,c_{1},c_{2}>0$ such that for all $\tau\geq \tau_{4}$, we have
\begin{equation}\label{est4.3.4}
 \| \nabla P \|^{2}_{L^{2}(\widehat Z)} +  |  (P-m_{\mathcal I}(P))|_{L^{2}(\widehat J_{1})}^{2} 
 \leq e^{c_{1}\tau }\left(\left\|  P \right\|^{2}_{L^{2}(\mathcal O_0)} +\left\|  \nabla P \right\|^{2}_{L^{2}(\mathcal O_0)} \right)
+e^{-c_{2}\tau} \left(\| U \|_{H^{2}(Z)} + | U_2 |_{H^{2}(J_1)}\right).
\end{equation}
On the other hand, we deduce from \eqref{estimatevelocity} and a Young inequality that
\begin{multline*}
\left\| U \right\|_{H^{1}\left(\widetilde Z\right)}+\left| U_2 \right|_{H^{1}\left(\widetilde J_1\right)}
 \leq C_{4}e^{-\frac{\mu_{3}c_{2}}{2(1-\mu_{3})}\tau} \left( \left\| U \right\|_{H^{1}\left(\widehat Z\right)}+\left| U_2 \right|_{H^{1}\left(\widehat J_1\right)}
		 \right)
\\
+e^{\frac{c_2}{2} \tau}\left( \left\| \nabla P \right\|_{L^{2}\left(\widehat Z\right)} 
 	+ \left| P_{|J_{1}}-m_{\mathcal I}(P)  \right|_{L^{2}(\widehat J_{1})}+ \left\| U \right\|_{L^{2}(\mathcal O_0)} 
\right)
\end{multline*}
and combining this relation with \eqref{est4.3.4}, we deduce the existence of $c_{3},c_{4}>0$ such that for all $\tau\geq \tau_{4}$
$$
\left\| U \right\|_{H^{1}\left(\widetilde Z\right)}+\left| U_2 \right|_{H^{1}\left(\widetilde J_1\right)}
 \lesssim e^{-c_3 \tau} \left( \left\| U \right\|_{H^{2}\left(Z\right)}+\left| U_2 \right|_{H^{2}\left(J_1\right)}\right)
\\
+e^{c_4 \tau}\left( \left\|  P \right\|^{2}_{L^{2}(\mathcal O_0)} +\left\|  \nabla P \right\|^{2}_{L^{2}(\mathcal O_0)} + \left\| U \right\|_{L^{2}(\mathcal O_0)} \right)
$$
Now, we can use $c_P$ in \eqref{DefofUandP} so that for all $s\in [0,S_0]$, 
$$
\int_{\omega_0} P(s,x) \ dx=0
$$
and using the Poincar\'e-Wirtinger inequality, we deduce that 
$$
 \left\|  P \right\|^{2}_{L^{2}(\mathcal O_0)}\lesssim  \left\| \nabla P \right\|^{2}_{L^{2}(\mathcal O_0)} \lesssim  \left\| \Delta_z U \right\|_{L^{2}(\mathcal O_0)}.
$$
We deduce that for some constants $c_{5},c_{6}>0$, for all $\tau\geq \tau_{4}$,
$$
\left\| U \right\|_{H^{1}\left(\widetilde Z\right)}+\left| U_2 \right|_{H^{1}\left(\widetilde J_1\right)}
 \lesssim e^{-c_5 \tau} \left( \left\| U \right\|_{H^{2}\left(Z\right)}+\left| U_2 \right|_{H^{2}\left(J_1\right)}\right)
\\
+e^{c_6 \tau}\left\| U \right\|_{H^{2}(\mathcal O_0)}.
$$
Optimizing this inequality with respect to $\tau\geq \tau_{4}$ (see, for instance, \cite[Lemma 8.4]{Buffe17}) 
allows us to conclude the proof of \cref{globalinterpolation}.
\end{proof}

\subsection{From the interpolation inequality to the spectral inequality}
Using \cref{globalinterpolation}, we are now in a position to prove \cref{SpectralInequality}. This inequality combined with \cref{T2} yields the main result of the article (\cref{Tmain}).
\begin{proof}[Proof of \cref{SpectralInequality}]
From \eqref{ns7.1} and \eqref{fs3.4}, we deduce that
$$
\left\| U \right\|_{H^{1}\left(\widetilde Z\right)}^2+\left| U_2 \right|_{H^{1}\left(\widetilde J_1\right)}^2
\geq 
\left\| U \right\|_{L^{2}\left(\widetilde Z\right)}^2+\left| U_2 \right|_{L^{2}\left(\widetilde J_1\right)}^2
\geq 
\int_{s_0- \frac{S_{0}}{10}}^{s_0+ \frac{S_{0}}{10}} \sum_{\lambda_{j}\leq \Lambda} |a_{j}|^2 \cosh(\sqrt \lambda_{j}s)^2 \ ds
\gtrsim 
 \sum_{\lambda_{j}\leq \Lambda} |a_{j}|^2
$$
and
$$
\left\| U \right\|_{H^{2}\left(Z\right)}^2+\left| U_2 \right|_{H^{2}\left(J_1\right)}^2
\lesssim
e^{C\sqrt{\Lambda}} \sum_{\lambda_{j}\leq \Lambda} |a_{j}|^2
$$
Combining \cref{globalinterpolation} with the previous relations, we deduce that
$$
\sum_{\lambda_{j}\leq \Lambda} |a_{j}|^2 \lesssim e^{C\sqrt{\Lambda}} \left\| U \right\|_{H^{2}(\mathcal O_0)}^2.
$$
Using a cut-off function and integrations by parts, we deduce 
$$
\sum_{\lambda_{j}\leq \Lambda} |a_{j}|^2 \lesssim e^{C\sqrt{\Lambda}} \left\| U \right\|_{L^{2}(\mathcal O)}^2.
$$
and thus \cref{SpectralInequality}.
\end{proof}

\appendix

\section{Proof of \cref{Carlemansurv2}}\label{Carlemansurv2proof}
\subsection{A Carleman estimate}
The proof of \cref{Carlemansurv2} is mainly based on a Carleman estimate obtained in \cite{Buffe17} that we recall here. 
We recall that $Z, J_1, J_0$ are defined by \eqref{defZJ} whereas 
$\widetilde{Z}$ and $\widetilde{J}_1$ are defined by \eqref{Ztilde}. 
In what follows, we consider $z^0\in \widetilde{J}_1$, an open neighborhood $V$ of $z^0$ in $\overline{Z}$ and
a weight function $\varphi\in C^{\infty}(\overline{V})$. For any $\sigma\in \mathbb{R}$, we define
\[
p_{\varphi,\sigma}(z,\xi,\tau) = |\xi|^{2} -\tau^{2}|\nabla_{z}\varphi (z)|^{2}-\sigma^{2}+2i\tau \xi\cdot\nabla_{z}\varphi(z), 
\quad (z\in V,\  \xi\in \mathbb{R}^3,\  \tau \in \mathbb{R}).
\]
It is the principal symbol of the conjugated operator associated with $-\Delta_z-\sigma^2$, that is, of the operator 
$$
P_{\varphi,\sigma}  = -e^{\tau\varphi}\left(\Delta_{z}+\sigma^2\right)e^{-\tau\varphi} 
= -\Delta_{z} +2\tau\nabla_{z}\varphi\cdot\nabla_{z}-\tau^{2}|\nabla_{z}\varphi|^{2} +\tau(\Delta_{z} \varphi )-\sigma^2.
$$
We assume the following hypotheses on $\varphi$: sub-ellipticity on $\overline{V}$, that is the existence of 
$\tau_0>0$ such that for any $z\in V$,  $\xi\in \mathbb{R}^3$, $|\sigma|\geq 1$ and $\tau\geq \tau_0|\sigma|$, 
\begin{equation}\label{ellipticitehormanderintro}
p_{\varphi,\sigma}(z,\xi,\tau)=0 \implies \frac{1}{2i}\left\{\overline{p_{\varphi,\sigma}},p_{\varphi,\sigma}\right\}(z,\xi,\tau)>0.
\end{equation}
and the two following conditions to handle Ventcel boundary conditions (see \cite[conditions (23) and (24)]{Buffe17}):
\begin{equation}\label{weightfunctionpropertiesintro}
\nabla_z\varphi\neq 0\quad\text{in} \ \overline{V},\quad\text{and}\quad 
\sup_{\overline V\cap\widetilde J_{1}} \left|\nabla_{s,x_{1}}\varphi \right| \leq \nu_{0} \inf_{\overline V} \left|\partial_{x_2} \varphi\right|,
\end{equation}
for $\nu_{0}>0$ small enough. We recall that the Poisson bracket is defined by
$$
\left\{p^{(1)},p^{(2)}\right\}=\sum_{j=1}^3 \frac{p^{(1)}}{\partial \xi_j} \frac{p^{(2)}}{\partial z_j}-\frac{p^{(2)}}{\partial \xi_j} \frac{p^{(1)}}{\partial z_j}
$$
where we set here $z_1=s$, $z_2=x_1$ and $z_3=x_2$ to simplify.

Then we have the following result proved in \cite{Buffe17}:
\begin{Theorem}\label{Carlemanestimate}
Assume $z^0\in \widetilde{J}_1$ and $V$ is an open neighborhood of $z^0$ in $\overline{Z}$. Assume also that
$\varphi\in C^{\infty}(\overline{V})$ satisfies the conditions \eqref{ellipticitehormanderintro}) and \eqref{weightfunctionpropertiesintro} for $\nu_0$ small enough.
Then, there exist $\tau_0>0$ and $C>0$ such that
for all $|\sigma|\geq1$, for all $\tau\geq\tau_{0}|\sigma|$ and for all $w\in C^{\infty}_{0}(V)$,
\begin{multline*}
\tau^3 \left\| e^{\tau\varphi} w \right\|_{L^2(V)}^{2}
+\tau \left\| e^{\tau\varphi}\nabla_z w \right\|_{L^{2}(V)}^{2}
+\tau^{3} \left| e^{\tau\varphi} w_{|J_{1}} \right|_{L^{2}\left(V\cap\widetilde J_{1}\right)}^{2}
+\tau \left| e^{\tau\varphi} \nabla_{s,x_{1}} w_{| J_{1}} \right|_{L^{2}\left(V\cap \widetilde J_{1}\right)}^{2}
\\
+\tau \left| e^{\tau\varphi}\partial_{x_{2}} w_{|J_{1}} \right|^2_{L^2\left(V\cap\widetilde J_{1}\right)}
\leq 
C\left(\left\| e^{\tau\varphi}(-\Delta_{z}-\sigma^{2}) w \right\|^{2}_{L^{2}(V)}
+\tau \left| e^{\tau\varphi} (\partial_{x_{2}}w_{|J_{1}}- \Delta_{s,x_1} w_{|J_{1}} \right|_{L^{2}\left(V\cap\widetilde J_{1}\right)}^{2}\right).
\end{multline*}
 \end{Theorem}
 First to precise the above statement, by $w\in C^{\infty}_{0}(V)$ we mean that $w$ is the restriction of a $C^\infty$ function with compact support in $V_0$ where $V_0$ is an open set of $\mathbb{R}\times \mathcal{I}\times \mathbb{R}$ such that $V_0\cap \overline{Z}=V$. Second, we use the above result in the case 
$\sigma=1$, so that by taking $\tau_0$ large enough, we obtain that
for all $\tau\geq\tau_{0}$ and for all $w\in C^{\infty}_{0}(V)$,
\begin{multline}\label{Carleman}
\tau^3 \left\| e^{\tau\varphi} w \right\|_{L^2(V)}^{2}
+\tau \left\| e^{\tau\varphi}\nabla_z w \right\|_{L^{2}(V)}^{2}
+\tau^{3} \left| e^{\tau\varphi} w_{|J_{1}} \right|_{L^{2}\left(V\cap\widetilde J_{1}\right)}^{2}
+\tau \left| e^{\tau\varphi} \nabla_{s,x_{1}} w_{| J_{1}} \right|_{L^{2}\left(V\cap \widetilde J_{1}\right)}^{2}
\\
+\tau \left| e^{\tau\varphi}\partial_{x_{2}} w_{|J_{1}} \right|^2_{L^2\left(V\cap\widetilde J_{1}\right)}
\leq 
C\left(\left\| e^{\tau\varphi}\Delta_{z} w \right\|^{2}_{L^{2}(V)}
+\tau \left| e^{\tau\varphi} \left(\partial_{x_{2}}w_{|J_{1}}- \Delta_{s,x_1} w_{|J_{1}} \right) \right|_{L^{2}\left(V\cap\widetilde J_{1}\right)}^{2}\right).
\end{multline}

\subsection{Interpolation estimates for the Ventcel boundary condition}
Using the Carleman inequality of the previous section, one can deduce, in a classical way, an interpolation inequality. 
First let us define the weight function that we are going to use.

We consider the following norms on $\mathbb{R}\times \mathcal{I}\times \mathbb{R}$:
$$
\left| \begin{pmatrix}
s, x_1, x_2
\end{pmatrix}\right|_{\lambda}:=
\left( \frac{s^2}{\lambda^2} + \frac{x_1^2}{\lambda^2} +x_2^2\right)^{1/2}.
$$
We consider $z^0=(s^*,x_1^*,1)\in \widetilde{J}_1$, 
$$
V:=\left(s^*-\delta,s^*+\delta\right)\times\left(x_1^*-\delta,x_1^*+\delta\right)\times \left(1-\delta,1\right],
$$
with $\delta\in (0,1)$ small enough such that 
$$
\left(s^*-\delta,s^*+\delta\right)\subset \left(s_0- \frac{S_{0}}{9},s_0+ \frac{S_{0}}{9}\right)
$$
(see \eqref{Zhat}). We also define $z^*=(s^*,x_1^*,0)$.
Then we define
$$
\psi(z):=\left| z-z^*\right|_{\lambda}, \quad \varphi = e^{-\lambda \psi}.
$$
\begin{Lemma}
There exists $\lambda_0>0$ such that for any $\lambda\geq \lambda_0$,
the weight function $\varphi$ satisfies \eqref{ellipticitehormanderintro} and \eqref{weightfunctionpropertiesintro} on $V$ 
for some $\tau_{0} = \tau_0(\lambda)$.
\end{Lemma} 
\begin{proof}
We assume that $\lambda\geq 1$ in all what follows.
First, since $\nabla_z \psi\neq 0$ in $\overline V$, we deduce the first point of \eqref{weightfunctionpropertiesintro}.
For the second point of \eqref{weightfunctionpropertiesintro}, we first notice that
\begin{equation}\label{23:37}
\inf_{\overline V\cap\widetilde J_{1}} \psi=1, \quad 
\sup_{\overline V} \psi=\left(\frac{2 \delta^2}{\lambda^2} +1\right)^{1/2} \leq 2.
\end{equation}
We thus deduce
$$
\sup_{\overline V\cap\widetilde J_{1}} \left|\nabla_{s,x_{1}}\varphi \right|
\leq
\sup_{\overline V\cap\widetilde J_{1}} \varphi =e^{-\lambda}
\quad
\text{and}
\quad 
\inf_{\overline V} \left|\partial_{x_2} \varphi\right|\geq \lambda \frac{1-\delta}{2} e^{-\left(2 \delta^2+\lambda ^2\right)^{1/2}}.
$$
Consequently there exists $\lambda_1$ such that the second point of \eqref{weightfunctionpropertiesintro} holds for $\lambda \geq \lambda_1$.

For \eqref{ellipticitehormanderintro}, we compute the Poisson bracket:
\begin{align*}
\frac{1}{8\tau i}\{\overline{p_{\varphi,\sigma}},p_{\varphi,\sigma}\}
& 
=\tau^{2} \left(\nabla_z^2 \varphi\right)\nabla_{z}\varphi \cdot \nabla_{z}\varphi+\left(\nabla_z^2 \varphi\right)\xi \cdot \xi
\\
& = \tau^{2} \varphi^{3}\left(  \lambda^{4} \left| \nabla_{z}\psi \right|^{4}-\lambda^{3} \left(\nabla_z^2 \psi\right) (\nabla_{z}\psi)\cdot (\nabla_{z}\psi) \right)
+\varphi\left( \lambda^{2}(\nabla_{z}\psi\cdot\xi)^{2}-\lambda \left(\nabla_z^2 \psi\right)\xi \cdot\xi\right)\\
& \geq 
\tau^{2} \lambda^{4}\varphi^{3}\left|\nabla_{z}\psi \right|^{4}
	-\tau^{2} \lambda^{3} \varphi^{3} \left|\nabla_z^2 \psi\right| \left| \nabla_{z}\psi \right|^{2} 
	- \lambda \varphi \left|\nabla_z^2 \psi\right| |\xi|^{2}.
\end{align*}
Now, if $p_{\varphi,\sigma}(z,\xi,\tau)=0$, then $|\xi|^{2} =\tau^{2}\lambda^2 \varphi^2 |\nabla_{z}\psi (z)|^{2}+\sigma^{2}$ so that
$$
\frac{1}{8\tau i}\{\overline{p_{\varphi,\sigma}},p_{\varphi,\sigma}\}
\geq 
\tau^{2} \lambda^{4}\varphi^{3}\left|\nabla_{z}\psi \right|^{4}
	-2\tau^{2} \lambda^{3} \varphi^{3} \left|\nabla_z^2 \psi\right| \left| \nabla_{z}\psi \right|^{2} 
	- \lambda \varphi \left|\nabla_z^2 \psi\right| \sigma^{2}.
$$
From \eqref{23:37}, there exist positive constants independent of $\lambda$ such that
$$
C_1\leq \left|\nabla_{z}\psi \right| \leq C_2, \quad \left|\nabla_z^2 \psi\right|\leq C_3.
$$
In particular there exist $C>0$ and $\lambda_0\geq \lambda_1$, such that for $\lambda\geq \lambda_0$,
$$
\frac{1}{8\tau i}\{\overline{p_{\varphi,\sigma}},p_{\varphi,\sigma}\}
\geq 
C \tau^{2} \lambda^{4}\varphi^{3}- \lambda \varphi \left|\nabla_z^2 \psi\right| \sigma^{2}
$$
and there exists $\tau_0=\tau_0(\lambda)$ such that for $\tau\geq \tau_0 |\sigma|$, 
$$
\frac{1}{8\tau i}\{\overline{p_{\varphi,\sigma}},p_{\varphi,\sigma}\}\geq 
\frac{C}{2} \tau^{2} \lambda^{4}\varphi^{3}>0.
$$
\end{proof}
From now on, the value of $\lambda$ shall be kept fixed.
We define, for $\beta>0$,
\[
\widehat Z_{\beta} := \{z\in \widehat Z \ ; \ \dist(z,\widehat J_{1})>\beta\}.
\]

\begin{Lemma}\label{boundarylemma}
Assume $z^0\in \widetilde{J}_1$. There exist an open neighborhood $\widetilde{V}$ of $z^0$ in $\overline{Z}$, $\mu,\beta \in (0,1)$ and $C>0$ such that
for any $v\in H^{2}(Z)$ with $v_{|_{J_{1}}}\in H^{2}(J_{1})$,
\begin{multline}
 \left\| v \right\|_{H^1(\widetilde{V})}
+\left|  v_{|J_{1}} \right|_{H^{1}\left(\widetilde V\cap\widetilde J_{1}\right)}
\leq C
\left( \left\| v \right\|_{H^{1}\left(\widehat Z\right)}
+\left| v_{|J_{1}}\right|_{H^{1}\left(\widehat J_{1}\right)}
\right)^{1-\mu}
\\
\times
\left(
\left\| \Delta_{z} v \right\|_{L^{2}(\widehat{Z})}
+\left| \partial_{x_{2}}v_{|J_{1}}- \Delta_{s,x_1} v_{|J_{1}} \right|_{L^{2}\left(\widehat J_{1}\right)}
+ \left\| v \right\|_{H^{1}\left(\widehat{Z}_\beta\right)}
\right)^{\mu}.
\end{multline}
\end{Lemma}
\begin{proof}
Standard computation shows the existence of $r_3>1$ such that
$$
\left| z-z^*\right|_{\lambda}=r_3 \quad z=(s,x_1,1)\in J_1 \implies (s,x_1)\in 
	\left(s^*-\frac{\delta}{2}, s^*+\frac{\delta}{2}\right)\times \left(x_1^*-\frac{\delta}{2}, x_1^*+\frac{\delta}{2}\right).
$$
We consider $r_2\in (1,r_3)$. We can also check the existence of $r_1\in (0,\delta)$ such that
$$
\{z=(s,x_1,x_2)\in Z \ ; \ 1-r_{1}\leq x_{2}\leq 1, \quad \left| z-z^*\right|_{\lambda} \leq r_3\}\subset V.
$$
We consider two cut-off functions $\chi_{0},\chi_{1}\in C^{\infty}(\mathbb{R}\times \mathcal{I}\times \mathbb{R})$
such that
\begin{equation*}
\chi_{0}(z)=
\left\{\begin{array}{lll}
	 0&\text{if}&0<x_{2}<1-r_{1}\\
	 1&\text{if}&1-r_{1}/2<x_{2}<1,
\end{array}\right.
\quad
\chi_{1}(z)=
\left\{\begin{array}{lll}
	 0&\text{if}& \left| z-z^*\right|_{\lambda}>r_3\\
	 1&\text{if}& \left| z-z^*\right|_{\lambda}\leq r_2.
\end{array}\right.
\end{equation*}

Let us consider $v\in C^{\infty}(Z)$ and let us apply the Carleman estimate \eqref{Carleman}
to $w=\chi_{0}\chi_{1}v\in C^{\infty}_0(V)$.
In the right-hand side of this estimate, we have
$$
\Delta_{z} \left(\chi_{0}\chi_{1}v\right)
=\left(\chi_{0}\chi_{1}\right) \Delta_{z} v + 2\nabla_z \left(\chi_{0}\chi_{1}\right) \cdot \nabla_z v+ v \Delta_z \left(\chi_{0}\chi_{1}\right) 
$$
Note that in $\supp \chi_{0}\chi_{1}$, $\varphi \leq C_3:=e^{-\lambda (1-r_1)}$. The two last terms in the above relation are included in
$$
V\cap \left(\supp \nabla \chi_0\right)\subset V\cap \left\{x_2\in[1-r_1,1-r_1/2]\right\}
$$
and on this set, $\varphi\leq C_3$ or in 
$$
V\cap \left(\supp \nabla \chi_1\right)\subset V\cap \left\{r_2\leq \left| z-z^*\right|_{\lambda}\leq r_3\right\}
$$
and on this set, $\varphi\leq C_1:=e^{-\lambda r_2}$.
Therefore, 
\begin{equation}\label{ns9.0}
\left\| e^{\tau\varphi}\Delta_{z} w \right\|_{L^{2}(V)}
\lesssim 
e^{C_3 \tau} \left\| \Delta_{z} v \right\|_{L^{2}(V)}
+e^{C_3 \tau} \left\| v \right\|_{H^{1}\left(V\cap \left\{x_2\in[1-r_1,1-r_1/2]\right\}\right)}
+e^{C_1 \tau} \left\| v \right\|_{H^{1}\left(V\right)}
\end{equation}
and similarly,
\begin{equation}\label{ns9.1}
\left| e^{\tau\varphi} \left(\partial_{x_{2}}w_{|J_{1}}- \Delta_{s,x_1} w_{|J_{1}}\right) \right|_{L^{2}\left(V\cap\widetilde J_{1}\right)}
\lesssim 
e^{C_3 \tau} 
\left| \partial_{x_{2}}v_{|J_{1}}- \Delta_{s,x_1} v_{|J_{1}} \right|_{L^{2}\left(V\cap\widetilde J_{1}\right)}
+e^{C_1 \tau} 
\left| v_{|J_{1}}\right|_{H^{1}\left(V\cap\widetilde J_{1}\right)}.
\end{equation}
There exists $r_4>0$ such that 
$$
\left\{z\in \mathbb{R}\times \mathcal{I}\times \mathbb{R} \ ; \ \left|z-z^0\right|\leq r_4 \right\}
\subset 
\left\{
z=(s,x_1,x_2)\in \mathbb{R}\times \mathcal{I}\times \mathbb{R} \ ; \ 1-\frac{r_{1}}{2} < x_{2}, \quad \left| z-z^*\right|_{\lambda} < r_2
\right\}.
$$
Then on the set
$$
\widetilde{V}:=\left\{z\in Z \ ; \ \left|z-z^0\right|<r_4\right\} \subset V
$$
we have $\chi_0\chi_1=1$ and $\varphi\geq C_2:=e^{-\lambda \sup_{\widetilde{V}} \psi}$, with $C_2\in (C_1,C_3)$.
Combining this with \eqref{Carleman}, \eqref{ns9.0} and \eqref{ns9.1}, we deduce
that for all $\tau\geq \tau_0$,
\begin{multline}
 \left\| v \right\|_{H^1(\widetilde{V})}
+\left|  v_{|J_{1}} \right|_{H^{1}\left(\widetilde V\cap\widetilde J_{1}\right)}
\\
\lesssim
e^{\left(C_3-C_2\right) \tau} 
\left(
\left\| \Delta_{z} v \right\|_{L^{2}(V)}
+\left| \partial_{x_{2}}v_{|J_{1}}- \Delta_{s,x_1} v_{|J_{1}} \right|_{L^{2}\left(V\cap\widetilde J_{1}\right)}
+ \left\| v \right\|_{H^{1}\left(V\cap \left\{x_2\in[1-r_1,1-r_1/2]\right\}\right)}
\right)
\\
+e^{-\left(C_2-C_1\right) \tau}\left( \left\| v \right\|_{H^{1}\left(V\right)}
+\left| v_{|J_{1}}\right|_{H^{1}\left(V\cap\widetilde J_{1}\right)}
\right).
\end{multline}
Optimizing this inequality with respect to $\tau$ yields the interpolation inequality for $v$ smooth. A density argument permits to conclude the proof of 
\cref{boundarylemma}.
\end{proof}

\begin{proof}[Proof of Theorem \ref{Carlemansurv2}]
By a compactness argument, one can deduce from Lemma \ref{boundarylemma} an interpolation result on a neighborhood of $\widetilde J_1$. Then we combine this
with classical interpolation estimates (see \cite{LebeauRobbiano})
 in the interior and at the boundary $J_{0}$ where Dirichlet boundary condition hold to conclude. A similar proof is done in \cite{Buffe17} (see Lemma 8.3).
\end{proof}.


\bibliographystyle{plain}
\bibliography{references}

\end{document}